\documentclass[final,1p,times,authoryear]{elsarticle}
\usepackage{amsmath}
\usepackage{amsthm}
\usepackage{amssymb}
\usepackage{amsfonts}

\newtheorem{theorem}{Theorem}
\newtheorem{lemma}[theorem]{Lemma}
\newtheorem{corollary}[theorem]{Corollary}

\newtheorem{proposition}[theorem]{Proposition}
\newtheorem{definition}[theorem]{Definition}

\newtheorem{example}[theorem]{Example}
\newtheorem{remark}[theorem]{Remark}

\newtheorem{algorithm}[theorem]{Algorithm}

\begin{document}

\begin{frontmatter}

\title{Gr\"{o}bner--Shirshov bases for commutative dialgebras$^*$}

\author{Yuqun Chen}
\address{School of Mathematical Sciences, South China Normal University, Guangzhou 510631, P. R. China}
\ead{yqchen@scnu.edu.cn}
\ead{$^*$Supported by the NNSF of China (11571121), the NSF of Guangdong Province (2017A030313002) and the Science and Technology Program of Guangzhou (201707010137)}

\author{Guangliang Zhang}
\address{School of Mathematics and Systems Science, Guangdong Polytechnic Normal University, Guangzhou 510631, P. R. China}
\ead{zgl541@163.com}

\begin{abstract}
We establish Gr\"{o}bner--Shirshov bases theory for commutative dialgebras.
We show that for any ideal $I$ of $Di[X]$, $I$ has a unique reduced Gr\"{o}bner--Shirshov basis, where $Di[X]$ is the free commutative dialgebra generated by a set $X$, in particular, $I$ has a finite Gr\"{o}bner--Shirshov basis if $X$ is finite. As applications,
we give normal forms of elements of an arbitrary commutative disemigroup,
prove that the word problem for finitely presented commutative dialgebras (disemigroups) is solvable, and
show that if  $X$ is finite, then the problem
whether two ideals  of $Di[X]$ are identical is solvable.
We construct a Gr\"{o}bner--Shirshov basis in associative dialgebra $Di\langle X\rangle$ by
lifting a Gr\"{o}bner--Shirshov basis in $Di[X]$.
\end{abstract}

\end{frontmatter}

\section{Introduction}

Recently the study of algebraic properties of dialgebras  has attracted considerable attention.
Dialgebras are vector spaces over a field equipped with two binary bilinear associative operations
satisfying some axioms. Moreover, if operations of a dialgebra coincide,
we obtain an associative algebra and so, dialgebras are a generalization of associative algebras.
The class of dialgebras is rather interesting despite the lack of simple examples distinct from the associative algebras.
It is well-known that for Lie algebras there is a notion of a universal enveloping associative algebra.
By the Poincar$\acute{e}$-Birkhoff-Witt theorem,
to a given Lie algebra $L$ there exists an associative algebra $A$
such that $L$ is isomorphic to a subalgebra of the Lie algebra $A^{(-)}$.
The universal enveloping algebra for Leibniz
algebras, which are a non-commutative variation of Lie algebras, was found in \cite{Lo93,Lo95}.
Dialgebras serve as these enveloping algebras.
In \cite{Di}, a Composition-Diamond lemma for dialgebras was given to obtain
normal forms for some dialgebras including the universal enveloping algebra for Leibniz
algebras. \cite{Po} studied the connection of Rota-Baxter algebras and dialgebras with associative bar-unity.
\cite{Ko} proved that each dialgebra may be obtained in turn from an associative conformal algebra.
Lately normal forms of free commutative dialgebras were found by \cite{Zhuchok}, \cite{CZ}.

Gr\"{o}bner bases and Gr\"{o}bner--Shirshov bases were invented independently by A.I. Shirshov
for ideals of free (commutative, anti-commutative) non-associative algebras in \cite{Sh62b,Shir3},
free Lie algebras in \cite{Shir3} and implicitly free associative algebras in \cite{Shir3} (see also \cite{be78,bo76}), by
\cite{Hi64} for ideals of the power series algebras (both formal and convergent),
and by \cite{bu70} for ideals of the polynomial algebras. Gr\"{o}bner bases and
Gr\"{o}bner--Shirshov bases theories have been proved to be very useful in different branches
of mathematics, including commutative algebra and combinatorial algebra. It is a powerful
tool to solve the following classical problems: normal form; word problem; conjugacy
problem; rewriting system; automaton; embedding theorem; PBW theorem; extension;
homology; growth function; Dehn function; complexity; etc. See, for example, the books by
\cite{AL,BKu94,BuCL,BuW,CLO,Ei} and the surveys by \cite{BC,BokutChenBook,BoFKK00,BK03,BK05a,bo05}.

In Gr\"{o}bner--Shirshov bases theory for a category of algebras, a key part is to establish
``Composition-Diamond lemma" for such algebras. The name ``Composition-Diamond
lemma" combines the Diamond Lemma in \cite{Nm}, the Composition Lemma in \cite{Sh62b}
 and the Diamond Lemma in \cite{be78}.

In this paper, we  establish Gr\"{o}bner--Shirshov bases theory for commutative dialgebras.
A Composition Diamond lemma for commutative dialgebras is given, see Theorem \ref{cd2}.
We show that for a given monomial-center ordering, each ideal of $Di[X]$ has a unique reduced Gr\"{o}bner--Shirshov basis;
if $X$ is finite, then $Di[X]$ is Noetherian and each ideal of $Di[X]$ has a finite Gr\"{o}bner--Shirshov basis, and an algorithm is given to find such a finite (reduced) Gr\"{o}bner--Shirshov basis. As applications, we give normal forms of elements of an arbitrary commutative disemigroup and prove that for finitely presented commutative dialgebras (disemigroups), the word problem and the problem
whether two ideals  of $Di[X]$ are identical are solvable.
These results will be applied to computer algebra systems,
which are referring to the common use of the Buchberger approach to polynomials.
Moreover, we
prove a theorem on the pair of algebras ($Di[X],Di\langle X\rangle$) following the spirit of Eisenbud-Peeva-Sturmfels' theorem
on the pair ($k[X],k\langle X\rangle$) in \cite{EiPS}. Namely, we construct a Gr\"{o}bner--Shirshov
basis in associative dialgebra $Di\langle X\rangle$ by lifting a given Gr\"{o}bner--Shirshov basis $S$ in commutative dialgebra $Di[X]$.

The paper is organized as follows. In section 2, we review the free commutative dialgebra $Di[X]$ generated by $X$ over a field $\mathbf{k}$.
In section 3, we establish a Composition-Diamond lemma for commutative dialgebras.
In section 4, it is shown that each ideal of a finitely generated polynomial dialgebra has a finite Gr\"{o}bner--Shirshov basis.
Section 5 gives normal forms of commutative disemigroups
and shows the word problem for finitely presented commutative dialgebras (disemigroups) is solvable.
The main results in this section are similar to ones in Gr\"{o}bner bases theory for commutative algebras \citep{bu65,bu70,BuCL,BuW}.
Section 6 provides a method by which we can lift commutative Gr\"{o}bner--Shirshov bases to associative ones.

\ \

\section{Free commutative dialgebras}

Throughout the paper, we fix a field $\mathbf{k}$.
$\mathbb{Z}^+$ stands for the set of positive integers.

\begin{definition}\citep{Lo99}\label{l2}
A \emph{disemigroup (dialgebra)} is a set ($\mathbf{k}$-linear space) $D$ equipped with two maps
$$
\vdash \ : D\times D\rightarrow D,
\ \
\dashv \ : D\times D\rightarrow D,
$$
where $\vdash$ and $\dashv$ are associative and satisfy the following identities:
for all $a,  b,  c\in D,$
\begin{equation}\label{eq00}
\begin{cases}
a\dashv(b\vdash c)=a\dashv (b\dashv c),\\
(a\dashv b)\vdash c=(a\vdash b)\vdash c, \\
a\vdash(b\dashv c)=(a\vdash b)\dashv c.
\end{cases}
%\end{displaymath}
\end{equation}

A disemigroup (dialgebra) $(D,\vdash,\dashv)$ is  \emph{commutative} if both  $\vdash$ and $\dashv$ are commutative.

\end{definition}

Let $X=\{x_i|i\in I\}$ be a total-ordered set, $X^+\ (X^*)$ the free semigroup (monoid) generated by $X$,
$$
\lfloor X^+\rfloor:= \{\lfloor x_{i_1}x_{i_2}\cdots  x_{i_n}\rfloor\mid  i_1,\dots,i_n\in I, x_{i_1}\leq x_{i_2}\leq\cdots \leq x_{i_n}, n\in \mathbb{Z}^+\},
$$
the free commutative semigroup without unit generated by $X$, and
$$
\lfloor X^*\rfloor:=\lfloor X^+\rfloor\cup \{\varepsilon\},
$$
the free commutative monoid  generated by $X$, where $\varepsilon$ is the empty word.

For any
$u=x_{j_1} x_{j_2}\cdots x_{j_n}\in X^+,\ x_{j_k}\in X$, we define
$$
\lfloor u \rfloor=\lfloor x_{j_1}x_{j_2}\cdots  x_{j_n}\rfloor:
=\lfloor x_{i_1} x_{i_2}\cdots  x_{i_n}\rfloor, \ x_{i_1}\leq x_{i_2}\leq\dots \leq x_{i_n},
$$
where $\{x_{i_1},x_{i_2},\dots, x_{i_n}\}=\{x_{j_1},x_{j_2},\dots, x_{j_n}\}$ as multisets.

Write
\begin{eqnarray*}
\lfloor X^+ \rfloor_{_1}:&=&\{\lfloor u \rfloor_1 \mid \lfloor u\rfloor\in \lfloor X^+\rfloor\}, \\
\lfloor XX \rfloor_{_{2}}:&=&\{\lfloor v \rfloor_2 \mid \lfloor v\rfloor\in \lfloor X^+\rfloor, |v|=2\}=\{\lfloor xx' \rfloor_2 \mid x,x'\in X,\ x\leq x'\},
\end{eqnarray*}
where $|v|$ is the length of $v$.
For any $h=\lfloor u \rfloor_p \in \lfloor X^+ \rfloor_{_1}\cup \lfloor XX \rfloor_{_{2}}$,
we call $\lfloor u\rfloor$ the \textit{associative (commutative) word} of $h$.
For convenience, we denote $\lfloor u \rfloor_1=u$ if $u\in X$.

\begin{lemma}\citep{Zhuchok,CZ}
Let $X=\{x_i\mid i\in I\}$ be a total-ordered set and
$$
Disgp[X]:=( \lfloor X^+ \rfloor_{_1}\cup \lfloor XX \rfloor_{_{2}}, \ \vdash, \dashv).
$$
 Then $Disgp[X]$ is the free commutative disemigroup generated by $X$,
where the operations $\vdash$ and $\dashv$ are as follows: for any
$x,x'\in X,\ \lfloor u\rfloor_{p_1},\ \lfloor v\rfloor_{p_2}\in \lfloor X^+ \rfloor_{_1}\cup \lfloor XX \rfloor_{_{2}}$ with  $|uv|>2$,
\begin{eqnarray*}
&&\lfloor v\rfloor_{p_2}\vdash \lfloor u\rfloor_{p_1}
=\lfloor u\rfloor_{p_1}\vdash \lfloor v\rfloor_{p_2}=\lfloor u\rfloor_{p_1}\dashv \lfloor v\rfloor_{p_2}
=\lfloor v\rfloor_{p_2}\dashv \lfloor u\rfloor_{p_1}=\lfloor uv\rfloor_1,\\
&&x\dashv x'=x'\dashv x=\lfloor xx'\rfloor_1, \\
&&x\vdash x'=x'\vdash x=\lfloor xx'\rfloor_2.
\end{eqnarray*}

Let $Di[X]$ be the $\mathbf{k}$-linear space with a $\mathbf{k}$-basis $\lfloor X^+ \rfloor_{_1}\cup \lfloor XX \rfloor_{_{2}}$.
Then $(Di[X], \ \vdash, \dashv )$ is the
free commutative dialgebra generated by $X$.
\end{lemma}
For example, if $\lfloor u \rfloor,\lfloor v \rfloor\in \lfloor X^+\rfloor, \
\lfloor u \rfloor=\lfloor x_{i_1}x_{i_2}\cdots  x_{i_n} \rfloor,\lfloor v \rfloor=\lfloor x_{j_1}x_{j_2}\rfloor$, $x_{i_l},x_{j_k}\in X$, then with the notation as in \cite{Di,CZ},
$$
\lfloor u \rfloor_1:=\dot{x}_{i_1}x_{i_2}\cdots  x_{i_n}
=x_{i_1}\dashv x_{i_2}\dashv\cdots \dashv x_{i_n}, \ \ \lfloor v \rfloor_2:=x_{j_1}\dot{x}_{j_2}=x_{j_1}\vdash x_{j_2}.
$$

Let $X$ be a well-ordered set. We define the \textit{deg-lex ordering} on $\lfloor X^+\rfloor$ by the following:
for any $\lfloor u\rfloor=\lfloor x_{i_1}x_{i_2}\cdots x_{i_n}\rfloor,
\lfloor v\rfloor=\lfloor x_{j_1}x_{j_2}\cdots x_{j_m}\rfloor\in \lfloor X^+\rfloor$,
where $x_{i_l},x_{j_t}\in X$,
\begin{equation*}\label{equ0}
\lfloor u\rfloor>\lfloor v\rfloor \ \Leftrightarrow \ (|u|,x_{i_1},x_{i_2},\cdots, x_{i_n})>(|v|,x_{j_1},x_{j_2},\cdots, x_{j_m}) \ \mbox{lexicographically}.
\end{equation*}
An ordering $>$ on $\lfloor X^+\rfloor$ is said to be \textit{monomial} if $>$ is a well ordering and
for any $\lfloor u\rfloor,\lfloor v\rfloor,\lfloor w\rfloor\in \lfloor X^+\rfloor$,
$$
\lfloor u\rfloor>\lfloor v\rfloor \Rightarrow \lfloor uw\rfloor>\lfloor vw\rfloor.
$$
Clearly, the deg-lex ordering is monomial.

\section{Composition-Diamond lemma for commutative dialgebras }

Let $>$ be a monomial ordering on $\lfloor X^+\rfloor$.
We define the \textit{monomial-center ordering} $>_d$ on $\lfloor X^+ \rfloor_{_1}\cup \lfloor XX \rfloor_{_{2}}$ as follows.
For any $\lfloor u\rfloor_m,\lfloor v\rfloor_n\in \lfloor X^+ \rfloor_{_1}\cup \lfloor XX \rfloor_{_{2}}$,
\begin{equation}\label{equ0}
\lfloor u\rfloor_m>_d\lfloor v\rfloor_n \ \mbox{if} \  (\lfloor u\rfloor,m)>(\lfloor v\rfloor,n) \ \ \mbox{lexicographically}.
\end{equation}
In particular, if $>$ is the deg-lex ordering on $\lfloor X^+\rfloor$,
we call the ordering defined by $(\ref{equ0})$ the \textit{deg-lex-center ordering} on $\lfloor X^+ \rfloor_{_1}\cup \lfloor XX \rfloor_{_{2}}$.
For simplicity of notation, we write $>$ instead of $>_d$ when no confusion can arise.
It is clear that a monomial-center ordering is a well ordering on $\lfloor X^+ \rfloor_{_1}\cup \lfloor XX \rfloor_{_{2}}$. Such an ordering plays an important role in the sequel.

Here and subsequently, $>$ is a  monomial-center ordering on $\lfloor X^+ \rfloor_{_1}\cup \lfloor XX \rfloor_{_{2}}$ unless otherwise stated.

For any nonzero polynomial $f\in Di[ X]$,
$$
f=\alpha_1\lfloor u_1\rfloor_{m_1}+\cdots+\alpha_n\lfloor u_n\rfloor_{m_n},
$$
where each $0\neq \alpha_i\in \mathbf{k},
\ \lfloor u_i\rfloor_{m_i}\in \lfloor X^+ \rfloor_{_1}\cup \lfloor XX \rfloor_{_{2}}$ and
$\lfloor u_1\rfloor_{m_1}>\cdots>\lfloor u_n\rfloor_{m_n}$.
We write

\noindent $\bullet$ $supp(f):=\{\lfloor u_1\rfloor_{m_1},\cdots,\lfloor u_n\rfloor_{m_n}\}$;

\noindent $\bullet$ $\overline{f}:=\lfloor u_1\rfloor_{m_1}$, the \textit{leading monomial} of $f$;

\noindent $\bullet$ $lt(f):=\alpha_1\lfloor u_1\rfloor_{m_1}$, the \textit{leading term} of $f$;

\noindent $\bullet$
$lc(f):=\alpha_1$, the \textit{coefficient} of $\overline{f}$;

\noindent $\bullet$
$\widetilde{f}:=\lfloor u_1\rfloor$, the associative word of $\overline{f}$;

\noindent $\bullet$
$r\!_{_f}:=f-lt(f)$.

Note that $\widetilde{f}\in \lfloor X^+\rfloor$.
$f$ is called \textit{monic} if $lc(f)=1$.
For any nonempty subset $S$ of $Di[ X]$,
$S$ is\textit{ monic} if $s$ is monic for all $s\in S$.

For convenience we assume that $\lfloor u\rfloor>0$ for any $\lfloor u\rfloor\in \lfloor X^+\rfloor$ and $\widetilde{0}=0$, and $\lfloor u\rfloor_m>0$ for any $\lfloor u\rfloor_m\in \lfloor X^+ \rfloor_{_1}\cup \lfloor XX \rfloor_{_{2}}$.

\begin{definition}\label{def0}
A nonzero polynomial $f\in Di[ X]$ is \emph{strong} if $\widetilde{f}> \widetilde{r\!_{_f}}$.
\end{definition}

The proof of the following proposition follows from the Definition \ref{def0}.
\begin{proposition}\label{pnotstrong}
A  nonzero polynomial  $f\in Di[ X]$  is not strong if and only if
$$
f=\alpha_1\lfloor xx'\rfloor_2+\alpha_2\lfloor xx'\rfloor_1+g,
$$
where $0\neq\alpha_1,\alpha_2\in \mathbf{k}, \ x,x'\in X, \ \ g\in Di[X]$ and  $\overline{g}<\lfloor xx'\rfloor_1$.
\end{proposition}

It is easy to check that
$>$ on $\lfloor X^+ \rfloor_{_1}\cup \lfloor XX \rfloor_{_{2}}$ is compatible with operations $\vdash$ and $ \dashv$ in the following sense:
for any $\lfloor u\rfloor_m, \lfloor v\rfloor_n, \lfloor w\rfloor_k\in \lfloor X^+ \rfloor_{_1}\cup \lfloor XX \rfloor_{_{2}}$,
\begin{eqnarray*}{ }
\ \lfloor u\rfloor_m >\lfloor v\rfloor_n, \lfloor u\rfloor >\lfloor v\rfloor &\Rightarrow &
\lfloor u\rfloor_m\dashv \lfloor w\rfloor_k>\lfloor v\rfloor_n\dashv \lfloor w\rfloor_k,\\
&& \lfloor u\rfloor_m\vdash \lfloor w\rfloor_k>\lfloor v\rfloor_n\vdash \lfloor w\rfloor_k;\\
\ \lfloor u\rfloor_m >\lfloor v\rfloor_n,  \lfloor u\rfloor=\lfloor v \rfloor
&\Rightarrow&\lfloor u\rfloor=\lfloor v \rfloor=\lfloor xx' \rfloor,\ x,x'\in X,\ m=2,\ n=1,\ \mbox{ and }\\
&&\lfloor u\rfloor_m\dashv \lfloor w\rfloor_k=\lfloor v\rfloor_n\dashv \lfloor w\rfloor_k
=\lfloor u\rfloor_m\vdash \lfloor w\rfloor_k=\lfloor v\rfloor_n\vdash \lfloor w\rfloor_k=\lfloor uw\rfloor_1.
 \end{eqnarray*}

From this, it follows that
\begin{lemma}\label{r1}
Let $0\neq f\in Di[ X]$ and $\lfloor u\rfloor_m\in \lfloor X^+ \rfloor_{_1}\cup \lfloor XX \rfloor_{_{2}}$. Then
 \begin{eqnarray*}
\overline{(f\vdash \lfloor u\rfloor_m)}\leq\ \overline{f}\vdash \lfloor u\rfloor_m, \ \  \
\overline{(f\dashv \lfloor u\rfloor_m)}\leq\ \overline{f}\dashv \lfloor u\rfloor_m.
 \end{eqnarray*}
In particular, if $f$ is strong, then $\overline{(f\vdash \lfloor u\rfloor_m)}=\overline{f}\vdash \lfloor u\rfloor_m,$ and
$\overline{(f\dashv \lfloor u\rfloor_m)}= \overline{f}\dashv \lfloor u\rfloor_m$.
\end{lemma}

\begin{example}
Let $X=\{x_1,x_2,x_3\}$, $x_3>x_2>x_1$, $Char\mathbf{k}\neq2,3$ and $>$ be the deg-lex-center ordering on
$\lfloor X^+ \rfloor_{_1}\cup \lfloor XX \rfloor_{_{2}}$.
Let $f=2\lfloor x_2x_3\rfloor_2-2\lfloor x_2 x_3\rfloor_1+3\lfloor x_1 x_3\rfloor_2$. Then
$$
\overline{f}=\lfloor x_2x_3\rfloor_2,\ lt(f)=2\lfloor x_2x_3\rfloor_2,\ lc(f)=2,\ \widetilde{f}=\lfloor x_2x_3\rfloor, \
r\!_{_f}=-2\lfloor x_2 x_3\rfloor_1+3\lfloor x_1 x_3\rfloor_2.
$$
The polynomial $f$ is not strong since $\widetilde{f}=\lfloor x_2x_3\rfloor=\widetilde{r\!_{_f}}$ and
 \begin{eqnarray*}
\overline{(f \dashv x_1)}=\overline{( f\vdash x_1)}&=&\lfloor x_1 x_1 x_3\rfloor_1
<\lfloor x_1x_2 x_3\rfloor_1=\overline{f} \vdash x_1=\overline{f}\dashv x_1, \\
\overline{(r\!_{_f} \dashv x_1)}=\overline{(r\!_{_f} \vdash x_1)}&=&\lfloor x_1x_2 x_3\rfloor_1=\overline{r\!_{_f}}\vdash x_1=\overline{r\!_{_f}}\dashv x_1.
\end{eqnarray*}
\end{example}

Here and subsequently, $S$ denotes a monic subset of $Di[ X]$ unless otherwise stated.

\begin{definition}
Let $S$ be a monic subset of $Di[X]$. A polynomial $g\in Di[X]$ is called a \emph{normal $S$-polynomial} in $Di[X]$ if either $g\in S$ or $g$ is one of the following:
\begin{eqnarray*}
g=s\vdash x,\ \  \mbox{where}\ s\in S,\ |\widetilde{s}|=1,\ x\in X,
\end{eqnarray*}
or
\begin{eqnarray*}
g=s\dashv \lfloor a\rfloor_1,\ \  \  \mbox{ where}\ s\in S,\ s \mbox{ is strong},\  \lfloor a\rfloor \in \lfloor X^+\rfloor.
\end{eqnarray*}
If this is so, we also call $g$ a \emph{normal $s$-polynomial}.
\end{definition}

\begin{lemma}\label{lorder}
Let $s\in S$ and $g$ be a normal $s$-polynomial. Then

\begin{equation}\label{eq000}
%\begin{displaymath}
\overline{g}=
\begin{cases}
\overline{s}
& \text{if $g=s$,}\\
\lfloor\widetilde{s}x\rfloor_2  & \text{if $g=s\vdash x$,}\\
\lfloor\widetilde{s}a\rfloor_1  & \text{if $g=s\dashv \lfloor a\rfloor_1$.}
\end{cases}
%\end{displaymath}
\end{equation}
In particular, $\overline{g}\geq \overline{s}$, and $\overline{g}=\overline{s}$ if and only if $g=s$.
\end{lemma}

\begin{proof}
The proof of (\ref{eq000}) is straightforward.
It remains to prove that $\overline{g}> \overline{s}$ if $\overline{g}\neq \overline{s}$. Suppose that $\overline{g}\neq\overline{s}$. Then $\widetilde{g}=\lfloor\widetilde{s}c\rfloor$ where $\lfloor c\rfloor\in \lfloor X^+\rfloor$.
We claim that $\widetilde{g}>\widetilde{s}$. Otherwise, $\widetilde{s}>\widetilde{g}=\lfloor\widetilde{s}c\rfloor$. We have an infinite descending chain
 $$
 \widetilde{s} >\lfloor\widetilde{s}c\rfloor >\lfloor\widetilde{s}c^2\rfloor >\lfloor\widetilde{s}c^3\rfloor >\cdots,
$$
which contradicts the fact that $>$ is a monomial ordering on $\lfloor X^+\rfloor$. This clearly forces $\overline{g}>\overline{s}$.
\end{proof}

Note that if $g$ is a normal $s$-polynomial, then $\overline{g}=\lfloor\widetilde{s}b\rfloor_m$ for some $\lfloor b\rfloor\in \lfloor X^*\rfloor$ and $m\in \{1,2\}$.

From now on, we use $\lfloor sb\rfloor_m$ to present a normal $S$\!-polynomial, where $s\in S,\ \lfloor b\rfloor\in \lfloor X^*\rfloor$ and $m\in \{1,2\}$, i.e.
\begin{equation}\label{nsw}
%\begin{displaymath}
\lfloor sb\rfloor_m=
\begin{cases}
s
& \text{if $b=\varepsilon$,}\\
s\vdash b  & \text{if $|\widetilde{s}|=1$ and $b\in X$,}\\
s\dashv \lfloor b\rfloor_1  & \text{if $s$ is strong and $\lfloor b\rfloor \in \lfloor X^+\rfloor$.}
\end{cases}
%\end{displaymath}
\end{equation}
It follows immediately that $\overline{\lfloor sb\rfloor_m}=\lfloor \widetilde{s}b\rfloor_m$.

For $\lfloor a\rfloor,\lfloor b\rfloor\in \lfloor X^+\rfloor$, we denote the least common multiple of $\lfloor a\rfloor$ and $\lfloor b\rfloor$
by $lcm\{\lfloor a\rfloor,\lfloor b\rfloor\}$.

\begin{definition}\label{compo}
Let $f,g$ be two monic polynomials in $Di[X]$.
\begin{enumerate}
\item[(i)] If $f$ is not strong, then for any $x\in X$,
we call both $f\vdash x$ and $f\dashv x$ the \emph{multiplication compositions} of $f$.

\item[(ii)] If $f$ is strong, $|\widetilde{f}|>1$ and $supp(f)\cap X\neq \emptyset$, then for any $x\in X$,
we call $(f\vdash x)-(f\dashv x)$ the \emph{special composition} of $f$.

\item[(iii)] If $f\neq g$ and $\overline{f}=\overline{g}$, then
we call $(f,g)_{\overline{f}}=f-g$ the \emph{equal composition} of $f$ and $g$.

\item[(iv)] Suppose that $|\widetilde{f}|=2$ and $|\widetilde{g}|=1$.
If there exists a normal $g$-polynomial $\lfloor gx\rfloor_m$ for some $x\in X$ such that $\overline{f}=\overline{\lfloor gx\rfloor_m}$, then
we call $(f,g)_{\overline{f}}=f-\lfloor gx\rfloor_m$ the \emph{short intersection composition} of $f$ and $g$.

\item[(v)] Suppose that $f,g$ are strong, $|\widetilde{f}|+|\widetilde{g}|>3$ and $\overline{f}\neq \overline{g}$.
If $\widetilde{f}=\widetilde{g}$, then for any $x\in X$,
we call $(f,g)_{\lfloor \widetilde{f}x\rfloor_1}=\lfloor fx\rfloor_1-\lfloor gx\rfloor_1$ the \emph{equal multiplication composition} of $f$ and $g$;
if $\widetilde{f}\neq\widetilde{g}$ and
$lcm\{\widetilde{f},\widetilde{g}\}=\lfloor w\rfloor=\lfloor \widetilde{f}a\rfloor=\lfloor \widetilde{g}b\rfloor$ for some
$\lfloor a\rfloor \in \lfloor X^*\rfloor,\lfloor b\rfloor\in \lfloor X^+\rfloor$
and $|w|<|\widetilde{f}|+|\widetilde{g}|$, then we call
$(f,g)_{\lfloor w\rfloor_1}=\lfloor fa\rfloor_1-\lfloor gb\rfloor_1$
the \emph{long intersection composition} of $f$ and $g$.
\end{enumerate}
\end{definition}

\begin{definition}\label{dcgsb}
Let $S$ be  a monic subset of $Di[X]$.
A polynomial $h\in Di[X]$ is called \emph{trivial modulo} $S$,
if $h=\sum_i \alpha_{i}\lfloor s_ib_i\rfloor_{m_i}$, where each $\alpha_{i}\in \mathbf{k},
s_i\in S, \lfloor b_i\rfloor\in \lfloor X^*\rfloor$,
and $\overline{\lfloor s_ib_i\rfloor_{m_i}}\leq \overline{h}$ \emph{if} $\alpha_{i}\neq0$.

A monic set $S$ is called a \emph{Gr\"{o}bner--Shirshov basis} in $Di[X]$ if any
composition of polynomials in $S$ is trivial modulo $S$.

$S$  is said to be \emph{closed}  under the multiplication and special compositions
if any multiplication and special composition of polynomials in $S$ is trivial modulo $S$.
\end{definition}

For convenience, for any $f,g\in Di[X]$ and $\lfloor w\rfloor_m\in \lfloor X^+ \rfloor_{_1}\cup \lfloor XX \rfloor_{_{2}}$, we write
$$
f\equiv g  \mod(S,\lfloor w\rfloor_m)
$$
which means that$f-g=\sum_i \alpha_{i}\lfloor s_ib_i\rfloor_{m_i}$, where each $\alpha_{i}\in \mathbf{k},
s_i\in S, \lfloor b_i\rfloor\in \lfloor X^*\rfloor$,
and $\overline{\lfloor s_ib_i\rfloor_{m_i}}< \lfloor w\rfloor_m$ if $\alpha_{i}\neq0$.

We set
$$
Irr(S):=\{\lfloor u\rfloor_n\in  \lfloor X^+ \rfloor_{_1}\cup \lfloor XX \rfloor_{_{2}}\mid \lfloor u\rfloor_n\neq \overline{\lfloor sb\rfloor_m}\ \mbox{ for any normal } S\!\mbox{-polynomial}\ \lfloor sb\rfloor_m \}
$$
and use $Id(S)$ to denote the ideal of $Di[X]$ generated by $S$.

\begin{lemma}\label{l51}
Let $S$ be closed under the multiplication and special compositions, $\lfloor sb\rfloor_n$ a normal $S$\!-polynomial,
$\lfloor a\rfloor_m\in \lfloor X^+ \rfloor_{_1}\cup \lfloor XX \rfloor_{_{2}}$ and $f\in Id(S)$.
Then
\begin{enumerate}

\item[(i)] $\lfloor sb\rfloor_n \dashv \lfloor a\rfloor_m$ is trivial modulo $S$;

\item[(ii)] $\lfloor sb\rfloor_n\vdash \lfloor a\rfloor_m$ is trivial modulo $S$;

\item[(iii)]
$
f=\sum_i \alpha_{i}\lfloor s_ib_i\rfloor_{m_i},
$
where each $\alpha_{i}\in \mathbf{k}, s_i\in S, \lfloor b_i\rfloor\in \lfloor X^*\rfloor$.
\end{enumerate}
\end{lemma}

\begin{proof}
$(i)$ It suffices to show that $\lfloor sb\rfloor_n \dashv \lfloor a\rfloor_1$ is trivial modulo $S$.
The proof is by induction on $|a|$.
Suppose that $|a|=1$. The result holds trivially if $s$ is strong.
Assume that $s$ is not strong. Then $\lfloor b\rfloor$ is empty and we are done by the triviality of multiplication composition.
Now, let $|a|>1$ and write $\lfloor a\rfloor= y\lfloor a_1\rfloor$ in $X^*$, where $y\in X, \lfloor a_1\rfloor\in \lfloor X^+ \rfloor$.
Then, by the above arguments, $\lfloor sb\rfloor_n\dashv \lfloor a\rfloor_1=(\lfloor sb\rfloor_n\dashv y)\dashv \lfloor a_1\rfloor_1$ is a linear
combination of polynomials of the form $\lfloor s'c\rfloor_l\dashv \lfloor a_1\rfloor_1$,
where $s'\in S, \lfloor c\rfloor\in \lfloor X^*\rfloor$
and $\overline{\lfloor s'c\rfloor_l}\leq \overline{\lfloor sb\rfloor_n\dashv y}$.
By induction, $\lfloor s'c\rfloor_l\dashv \lfloor a_1\rfloor_1$ is a linear
combination of normal $S$\!-polynomials $\lfloor s_ib_i\rfloor_{m_i}$
and $\overline{\lfloor s_ib_i\rfloor_{m_i}}\leq\overline{\lfloor s'c\rfloor_l\dashv \lfloor a_1\rfloor_1}\leq\overline{\lfloor s'c\rfloor_l}\dashv \lfloor a_1\rfloor_1\leq
(\overline{\lfloor sb\rfloor_n\dashv y})\dashv \lfloor a_1\rfloor_1 =\overline{\lfloor sb\rfloor_n\dashv \lfloor a\rfloor_1}$,
which implies that $\lfloor sb\rfloor_n \dashv \lfloor a\rfloor_1$ is trivial modulo $S$.

$(ii)$ If $b\neq \varepsilon$ or $|a|>1$, then $\lfloor sb\rfloor_n\vdash \lfloor a\rfloor_m=\lfloor sb\rfloor_n \dashv \lfloor a\rfloor_1$
and we have done by $(i)$. It remains to prove that $s\vdash a$ is trivial modulo $S$, where $a\in X$.
Note that $s\vdash a$ is a normal $S$\!-polynomial if $|\widetilde{s}|=1$. Let $|\widetilde{s}|>1$.
If $s$ is not strong, then we are done by the triviality of multiplication composition.
Otherwise, $s\dashv a$ is a normal $S$\!-polynomial and $\overline{s\dashv a}=\overline{s\vdash a}$.
For $supp(s)\cap X=\emptyset$ it is easy to see that $s\vdash a=s\dashv a$ and we are done.
For $supp(s)\cap X\neq\emptyset$,
by the triviality of special composition, $s\vdash a-s\dashv a$
is a linear combination of normal $S$\!-polynomials $\lfloor s_ib_i\rfloor_{m_i}$
and $\overline{\lfloor s_ib_i\rfloor_{m_i}}\leq \overline{s\vdash a-s\dashv a}<\overline{s\vdash a}$.
Therefore the result holds.

$(iii)$ It is clear that $f$ is a linear combination of polynomials of the forms
$$
s, \ \ s\dashv \lfloor a\rfloor_m, \ \ s\vdash \lfloor a\rfloor_m,
$$
where $s\in S, \lfloor a\rfloor_m \in \lfloor X^+ \rfloor_{_1}\cup \lfloor XX \rfloor_{_{2}}$. Then the result follows from (i) and (ii).
\end{proof}

\begin{lemma}\label{l00}
Let $S$ be a monic subset of $Di[ X]$.
Then for any nonzero $f\in Di[ X]$,
$$
f=\sum_i\alpha_{i}\lfloor u_{i}\rfloor_{n_i}+\sum_j\beta_{j}\lfloor s_jb_j\rfloor_{m_j},
$$
where each $\lfloor u_i\rfloor_{n_i}\in Irr(S),\ \alpha_i, \beta_j\in
\mathbf{k},\ s_j\in S, \lfloor b_j\rfloor\in \lfloor X^*\rfloor$, $\lfloor u_i\rfloor_{n_i}\leq \overline{f}$ and
$\overline{\lfloor s_jb_j\rfloor_{m_j}}\leq\overline{f}$.
\end{lemma}

\begin{proof}
If $\overline{f}\in Irr(S)$,
then take $\lfloor u\rfloor_n=\overline{f}$ and $f_1=f-lc(f) \lfloor u\rfloor_n$. If $\overline{f}\notin Irr(S)$,
then $\overline{f}=\overline{\lfloor sb\rfloor_m}$ for some normal $S$\!-polynomial $\lfloor sb\rfloor_m$
and take $f_1=f-lc(f)\lfloor sb\rfloor_m$.
In both cases, we have $\overline{f_1}<\overline{f}$ and the result
follows from induction on $\overline{f}$.
\end{proof}

\begin{lemma}\label{lkey2}
Let $S$ be a Gr\"{o}bner--Shirshov basis in $Di[X]$, $\lfloor s_1b_1\rfloor_{m_1}$ and $\lfloor s_2b_2\rfloor_{m_2}$ normal $S$\!-polynomials.
If $\lfloor w\rfloor_m=\overline{\lfloor s_1b_1\rfloor_{m_1}}=\overline{\lfloor s_2b_2\rfloor_{m_2}}$, then
$$
\lfloor s_1b_1\rfloor_{m_1}-\lfloor s_2b_2\rfloor_{m_2}\equiv 0 \mod(S,\lfloor w\rfloor_m).
$$
\end{lemma}

\begin{proof} Since $\lfloor w\rfloor_m=\overline{\lfloor s_1b_1\rfloor_{m_1}}=\overline{\lfloor s_2b_2\rfloor_{m_2}}$,
it follows that $\lfloor w\rfloor=\lfloor \widetilde{s_1}b_1\rfloor=\lfloor\widetilde{s_2}b_2\rfloor$ and $m=m_1=m_2$.

If $b_1=b_2=\varepsilon$, then $\overline{s_1}=\overline{s_2}$ and we are done by the triviality of equal composition.

Suppose only one of $b_1$ and $b_2$ is empty, say, $b_1=\varepsilon$ and $b_2\neq \varepsilon$.
Then $\overline{s_1}=\overline{\lfloor s_2b_2\rfloor_{m_2}}$, $\lfloor b_2\rfloor\in \lfloor X^+\rfloor$ and $|\widetilde{s_1}|+|\widetilde{s_2}|\geq3$.
We thus have done by the triviality of short and long intersection composition.

Suppose that $b_1\neq\varepsilon$ and $b_2\neq \varepsilon$. Thus $s_1,s_2$ are strong. Here we need to consider two cases:

Case 1. $\widetilde{s_1}$ and $\widetilde{s_2}$ are mutually disjoint. We may assume that $\lfloor b_1\rfloor=\lfloor \widetilde{s_2}c\rfloor, \lfloor b_2\rfloor=\lfloor \widetilde{s_1}c\rfloor$,
where $\lfloor c\rfloor \in \lfloor X^*\rfloor$.
This splits into two cases depending on whether $m=1$ or $m=2$.
By (\ref{nsw}) and Lemmas \ref{l51} and \ref{r1},
\begin{displaymath}\lfloor s_1b_1\rfloor_{m_1}-\lfloor s_2b_2\rfloor_{m_2}=
\begin{cases}
\text{ $s_1\dashv \lfloor b_1\rfloor_1-s_2\dashv \lfloor b_2\rfloor_1
=s_1\dashv \overline{s_2}\dashv \lfloor c\rfloor_1-s_2\dashv \overline{s_1}\dashv \lfloor c\rfloor_1$}\\
\text{ $s_1\vdash \lfloor b_1\rfloor_1-s_2\vdash \lfloor b_2\rfloor_1
=s_1\vdash \overline{s_2}-s_2\vdash \overline{s_1},\ \ \ \  b_1,b_2\in X$}
\end{cases}
\end{displaymath}
\begin{displaymath}\ \ \ \ \ \ \ \ \ \ \ \ \ \ \ \ \ \ \ \ \ \ \ \ \ \ \ \ \ \ \ \ \ \ \ \ \ \ \ \ \ \ \ \ \ \ =
\begin{cases}
\text{ $s_1\dashv  (\overline{s_2}-s_2)\dashv \lfloor c\rfloor_1+s_2\dashv (s_1-\overline{s_1})\dashv \lfloor c\rfloor_1\equiv 0 \mod(S,\lfloor w\rfloor_1)$}\\
\text{ $s_1\vdash  (\overline{s_2}-s_2)+s_2\vdash (s_1-\overline{s_1})\equiv 0 \mod(S,\lfloor w\rfloor_2).$}
\end{cases}
\end{displaymath}

Case 2. $\widetilde{s_1}$ and $\widetilde{s_2}$ have a nonempty intersection.
In this case, we need to discuss three sub-cases:

Case 2.1. $|\widetilde{s_1}|+|\widetilde{s_2}|=2$.
Thus $\overline{s_1}=\overline{s_2}$ and $\lfloor b_1\rfloor=\lfloor b_2\rfloor$.
We have two possibilities depending on whether $m=1$ or $m=2$.
By the triviality of equal composition and Lemmas \ref{l51} and \ref{r1},
\begin{displaymath}\lfloor s_1b_1\rfloor_{m_1}-\lfloor s_2b_2\rfloor_{m_2}=
\begin{cases}
\text{ $s_1\dashv\lfloor b_1\rfloor_{1}-s_2\dashv\lfloor b_2\rfloor_{1}
=(s_1-s_2)\dashv \lfloor b_1\rfloor_1
\equiv 0 \mod(S,\lfloor w\rfloor_1)$}\\
\text{ $s_1\vdash\lfloor b_1\rfloor_{1}-s_2\vdash\lfloor b_2\rfloor_{1}
=(s_1-s_2)\vdash \lfloor b_1\rfloor_1
\equiv 0 \mod(S,\lfloor w\rfloor_2).$}
\end{cases}
\end{displaymath}

Case 2.2. $|\widetilde{s_1}|+|\widetilde{s_2}|=3$. Then $m_1=m_2=1$. We may assume that $|\widetilde{s_1}|=2$ and $|\widetilde{s_2}|=1$, i.e.
$\widetilde{s_1}=\lfloor\widetilde{s_2}x\rfloor$ and $\lfloor b_2\rfloor=\lfloor xb_1\rfloor$, where $x\in X$.
It follows that $\overline{s_1}=\overline{\lfloor s_2x\rfloor_l}$ for some $l=1,2$.
By the triviality of short composition and Lemmas \ref{l51} and \ref{r1},
$$
\lfloor s_1b_1\rfloor_{1}-\lfloor s_2b_2\rfloor_{1}
=s_1\dashv \lfloor b_1\rfloor_1-s_2\dashv \lfloor xb_1\rfloor_1=(s_1-\lfloor s_2x\rfloor_l)\dashv \lfloor b_1\rfloor_1
\equiv 0 \mod(S,\lfloor w\rfloor_1).
$$

Case 2.3. $|\widetilde{s_1}|+|\widetilde{s_2}|>3$. Then $m_1=m_2=1$. If $\widetilde{s_1}=\widetilde{s_2}$, then $\lfloor b_1\rfloor=\lfloor b_2\rfloor$.
We are done by the triviality of equal composition and equal multiplication composition, and Lemmas \ref{l51} and \ref{r1}.
Suppose that $\widetilde{s_1}\neq\widetilde{s_2}$ and
$lcm\{\widetilde{s_1},\widetilde{s_2}\}=\lfloor w_1\rfloor=\lfloor\widetilde{s_1}a_1\rfloor=\lfloor\widetilde{s_2}a_2\rfloor$, where $\lfloor a_1\rfloor, \lfloor a_2\rfloor\in \lfloor X^*\rfloor$. Then
$\lfloor b_1\rfloor=\lfloor a_1c\rfloor, \  \lfloor b_2\rfloor=\lfloor a_2c\rfloor$ for some $\lfloor c\rfloor\in \lfloor X^*\rfloor$ and $|w_1|<|\widetilde{s_1}|+|\widetilde{s_2}|$. Since  $\widetilde{s_1}\neq\widetilde{s_2}$, we may assume that $\lfloor a_2\rfloor\in \lfloor X^+\rfloor$. Then we have the long composition $\lfloor s_1a_1\rfloor_1-\lfloor s_2a_2\rfloor_1$.
By Lemmas \ref{l51} and \ref{r1},
$$
\lfloor s_1b_1\rfloor_{1}-\lfloor s_2b_2\rfloor_{1}=s_1\dashv\lfloor a_1c\rfloor_1-s_2\dashv\lfloor a_2c\rfloor_1=
(\lfloor s_1a_1\rfloor_1-\lfloor s_2a_2\rfloor_1)\dashv \lfloor c\rfloor_1
\equiv 0 \mod(S,\lfloor w\rfloor_1).
$$

The proof is complete.
\end{proof}

\ \

The following theorem is the main result in this article.

\begin{theorem}\label{cd2}
(Composition-Diamond lemma for commutative dialgebras) \ Let $S$ be a monic subset of $Di[X]$,
$>$ a monomial-center ordering on $\lfloor X^+ \rfloor_{_1}\cup \lfloor XX \rfloor_{_{2}}$ and
$Id(S)$ the ideal of $Di[X]$ generated by $S$. Then the following statements are equivalent.
\begin{enumerate}
\item[(i)] \ $S$ is a Gr\"{o}bner--Shirshov basis in $Di[X]$.
\item[(ii)] \ $0\neq f\in Id(S)\Rightarrow \overline{f}=\overline{\lfloor sb\rfloor_m}$ for some normal $S$\!-polynomial $\lfloor sb\rfloor_m$.
\item[(iii)] \ The set
$$
Irr(S)=\{\lfloor a\rfloor_n\in \lfloor X^+ \rfloor_{_1}\cup \lfloor XX \rfloor_{_{2}} \mid
\lfloor a\rfloor_n\neq \overline{\lfloor sb\rfloor_m}\
\mbox{ for any normal }S\!\mbox{-polynomial }\lfloor sb\rfloor_m\}
$$
is a $\mathbf{k}$-basis of the commutative dialgebra $Di[X|S]:=Di[X]/Id(S)$.
\end{enumerate}
\end{theorem}

\begin{proof} $(i)\Rightarrow (ii)$. Let $0\neq f\in Id(S)$. Then by
Lemma \ref{l51} $f$ has an expression
\begin{equation}\label{e4c}
f=\sum_i \alpha_i\lfloor s_ib_i\rfloor_{m_i},
\end{equation}
where each $0\neq\alpha_i\in \mathbf{k}, \ \lfloor b_i\rfloor\in \lfloor X^*\rfloor, \ s_i\in S$. Write
$\lfloor w_i\rfloor_{m_i}=\overline{\lfloor s_ib_i\rfloor_{m_i}}=\lfloor\widetilde{s_i}b_i\rfloor_{m_i}, i=1,2,\cdots$.
We may assume without loss of generality that
$$
\lfloor w_1\rfloor_{m_1}=\lfloor w_2\rfloor_{m_2}=\cdots=\lfloor w_l\rfloor_{m_l}>\lfloor w_{l+1}\rfloor_{m_{l+1}}\geq
\lfloor w_{l+2}\rfloor_{m_{l+2}}\geq\cdots .
$$

If $l=1$, then
$\overline{f}=\overline{\lfloor s_1b_1\rfloor_{m_1}}$ and
the result holds. Suppose that $l\geq 2$.
Then
$$
\lfloor w_1\rfloor_{m_1}=\overline{\lfloor s_1b_1\rfloor_{m_1}}=\overline{\lfloor s_2b_2\rfloor_{m_2}}.
$$
By Lemma \ref{lkey2}, we can rewrite the first two summands of (\ref{e4c}) in the form
\begin{eqnarray*}
\alpha_1\lfloor s_1b_1\rfloor_{m_1}+\alpha_2\lfloor s_2b_2\rfloor_{m_1}
&=&(\alpha_1+\alpha_2)\lfloor s_1b_1\rfloor_{m_1}+\alpha_2(\lfloor s_2b_2\rfloor_{m_1}-\lfloor s_1b_1\rfloor_{m_1}) \\
&=&(\alpha_1+\alpha_2)\lfloor s_1b_1\rfloor_{m_1}+\sum_j \alpha_2\beta_j\lfloor s'_jd_j\rfloor_{n_j},
\end{eqnarray*}
where each $\lfloor s'_jd_j\rfloor_{n_j}$ is a normal $S$\!-polynomial and $\overline{\lfloor s'_jd_j\rfloor_{n_j}}<\lfloor w_1\rfloor_{m_1}$.
Thus the result follows from induction on $(\lfloor w_1\rfloor_{m_1},l)$.

$(ii)\Rightarrow (iii)$. By Lemma \ref{l00}, the set $Irr(S)$ generates
$Di[X|S]$ as a linear space. On the other hand, suppose that
$h=\sum_i\alpha_i\lfloor u_i\rfloor_{l_i}=0$ in $Di[X|S]$, where
each $\alpha_i\in \mathbf{k}$, $\lfloor u_i\rfloor_{l_i}\in {Irr(S)}$. This means that
$h\in Id(S)$.
Then all $\alpha_i$ must be equal to zero. Otherwise, $\overline{h}=\lfloor u_j\rfloor_{l_j}$ for some $j$ which contradicts $(ii)$.

$(iii)\Rightarrow(i)$. Suppose that $h$ is a composition
of elements of $S$. Clearly, $h\in Id(S)$. By Lemma \ref{l00},
$$
h=\sum_i\alpha_{i}\lfloor u_{i}\rfloor_{n_i}+\sum_j\beta_{j}\lfloor s_jb_j\rfloor_{m_j},
$$
where each $\lfloor u_i\rfloor_{n_i}\in Irr(S),\ \alpha_i, \beta_j\in
\mathbf{k},\ \lfloor b_j\rfloor\in \lfloor X^*\rfloor,\ s_j\in S$, $\lfloor u_i\rfloor_{n_i}\leq \overline{h}$ and
$\lfloor\widetilde{s_j}b_j\rfloor_{m_j}\leq\overline{h}$.
Then $\sum\alpha_{i}\lfloor u_{i}\rfloor_{n_i}\in Id(S)$.
By $(iii)$, we have all $\alpha_{i}=0$ and
$h$ is trivial modulo $S$.
\end{proof}

\ \

\noindent{\bf Buchberger--Shirshov  algorithm} If a monic subset $S$ of $Di[X]$ is not a Gr\"{o}bner--Shirshov basis
then one can add to $S$ all nontrivial compositions. Continuing this process
repeatedly, we finally obtain a Gr\"{o}bner--Shirshov basis $S^{comp}$ that contains $S$
and generates the same ideal, $Id(S^{comp})=Id(S)$.

\begin{remark} In Theorem \ref{cd2}, if $\dashv\ =\ \vdash$, then $Di[X]=\mathbf{k}[X]$ is free commutative algebra generated by $X$ and Theorem \ref{cd2}
is Buchberger Theorem in \cite{bu65}.
\end{remark}

A Gr\"{o}bner--Shirshov basis $S$ in $Di[X]$ is \textit{minimal} if for any $s\in S$, $\overline{s}\in Irr(S \backslash \{s\})$.
A Gr\"{o}bner--Shirshov basis $S$ in $Di[X]$ is \textit{reduced} if for any $s\in S$, $supp(s)\subseteq Irr(S \backslash \{s\})$.
If $S$ is a Gr\"{o}bner--Shirshov basis in $Di[X]$, then we also call $S$  a Gr\"{o}bner--Shirshov basis for $I=Id(S)$.

\begin{lemma}\label{remark5.8}
Let $I$ be an ideal of $Di[X]$ and $S$ a monic subset of $I$. If, for all nonzero $f\in I$, there exists a normal $S$\!-polynomial $\lfloor sb\rfloor_m$ such that
$\overline{f}=\overline{\lfloor sb\rfloor_m}$, then $I=Id(S)$ and $S$ is a Gr\"{o}bner--Shirshov basis for $I$.
\end{lemma}

\begin{proof} Clearly, $Id(S)\subseteq I$. For any nonzero $f\in I$, $\overline{f}=\overline{\lfloor sb\rfloor_m}$
for some  normal $S$\!-polynomial $\lfloor sb\rfloor_m$. Thus $f_1=f-lc(f)\lfloor sb\rfloor_m\in I$ and $\overline{f_1}<\overline{f}$.
By induction on $\overline{f}$, $f_1\in Id(S)$. Hence $f=f_1+lc(f)\lfloor sb\rfloor_m\in Id(S)$.
This shows that $I= Id(S)$. Now the result follows from Theorem \ref{cd2}.
\end{proof}

\begin{lemma}\label{lpredu}
Let $S$ be a Gr\"{o}bner--Shirshov basis in $Di[X]$ and $f\in S$. If $\overline{f}\notin Irr(S \backslash \{f\})$,
then $S \backslash \{f\}$ is a Gr\"{o}bner--Shirshov basis for $Id(S)$.
\end{lemma}
\begin{proof}Let $S_1=S \backslash \{f\}$. By Lemma \ref{remark5.8}, we only need to show that for any non-zero $h\in Id(S)$
there exists a $g\in S_1$ and~$\overline{h}=\overline{\lfloor gb\rfloor_m}$ for some normal $g$-polynomial $\lfloor gb\rfloor_m$.
Indeed, $\overline{h}=\overline{\lfloor sa\rfloor_m}$ for some normal $S$\!-polynomial $\lfloor sa\rfloor_m$ by Theorem \ref{cd2}.
If $s\neq f$, then we are done.
Suppose that $s=f$ and $\overline{f}=\overline{\lfloor s_1c\rfloor_l}$ for some normal  $S\!_1$\!-polynomial $\lfloor s_1c\rfloor_l$.
When $f$ is not strong or $s_1$ is strong, we conclude that
$$
\overline{h}=\overline{s}=\overline{f}=\overline{\lfloor s_1c\rfloor_2}  \ \ \mbox{or} \ \
\overline{h}=\overline{\lfloor sa\rfloor_m}=\overline{\lfloor fa\rfloor_m}=\overline{\lfloor s_1ca\rfloor_m}
$$
and we have done.
When $f$ is strong and $s_1$ is not strong, by Proposition \ref{pnotstrong}, we have
$$
\overline{f}=\overline{s_1}=\lfloor x_ix_j\rfloor_2 \ \ \mbox{and} \ \  \overline{f-s_1}=\lfloor x_ix_j\rfloor_1
\ \ \mbox{for some} \ \  x_i,x_j\in X.
$$
As $f-s_1\in Id(S)$ and $\overline{f-s_1}<\overline{f}$, $\lfloor x_ix_j\rfloor_1=\overline{f-s_1}=\overline{\lfloor s'a'\rfloor_1}$
some normal  $S\!_1$\!-polynomial $\lfloor s'a'\rfloor_1$ by Theorem \ref{cd2}.
It follows that $s'\in S\!_1$ is strong and
\begin{displaymath}\overline{h}=\overline{\lfloor sa\rfloor_m}=\overline{\lfloor fa\rfloor_m}=
\begin{cases}
\text{ $\lfloor x_ix_j\rfloor_2=\overline{s_1}$}\\
\text{ $\lfloor x_ix_ja\rfloor_1=\overline{\lfloor s'a'a\rfloor_1} $}
\end{cases}
\end{displaymath}
for some normal $s'$-polynomial $\lfloor s'a'a\rfloor_m$. Therefore the result holds.
\end{proof}

\begin{lemma}\label{lmgsb}
Let $S$ be a monic subset of $Di[X]$ and $s\in S$.
Then $s$ has an expression $s=s'+s''$,
where $s',s''\in Di[X], \ supp(s')\subset Irr(S\backslash \{s\}), \ s''\in Id(S\backslash \{s\})$ and
for any $\lfloor a\rfloor_m\in supp(s')$, $\lfloor a\rfloor_m\leq \overline{s}$.
Moreover, if $S$ is a minimal Gr\"{o}bner--Shirshov basis in $Di[X]$, then for any $s\in S$, $\overline{s}=\overline{s'}$ and $s'$ is strong if $s$ is strong.
\end{lemma}

\begin{proof} Analysis similar to that in the proof of Lemma \ref{l00} shows the first claim.

If $S$ is a minimal Gr\"{o}bner--Shirshov basis, then we have $\overline{s}=\overline{s'}\in Irr(S\backslash \{s\})$ for any  $s\in S$. Recall $r\!_{_s}:=s-\overline{s}$.
It follows that $r\!_{_s}=r\!_{_{s'}}+s''$ and $\overline{r\!_{_s}}\geq\overline{r\!_{_{s'}}}$. Thus $\widetilde{s'}=\widetilde{s}>\widetilde{r\!_{_s}}\geq \widetilde{r\!_{_{s'}}}$
and the result holds.
\end{proof}

Let $S$ be a subset of $Di[X]$ and $\lfloor a\rfloor_m\in \lfloor X^+ \rfloor_{_1}\cup \lfloor XX \rfloor_{_{2}}$. We set
$$
\overline{S}:=\{\overline{s} \mid s\in S\}, \ \ S^{\lfloor a\rfloor_m}:=\{s\in S\mid \overline{s}=\lfloor a\rfloor_m\}, \ \
S^{<{\lfloor a\rfloor_m}}:=\{s\in S\mid \overline{s}<\lfloor a\rfloor_m\}.
$$

\begin{theorem}\label{trgsb}
Let $I$ be an ideal of $Di[X]$ and  $>$ a monomial-center ordering on $\lfloor X^+ \rfloor_{_1}\cup \lfloor XX \rfloor_{_{2}}$.
Then there exists a unique reduced Gr\"{o}bner--Shirshov basis for $I$.
\end{theorem}

\begin{proof} It is clear that $\{lc(f)^{-1}f\mid 0\neq f\in I\}$ is a Gr\"{o}bner--Shirshov basis for $I$.
Let $S$ be an arbitrary Gr\"{o}bner--Shirshov basis for $I$.
For any $g\in S$, we set
\begin{eqnarray*}
\triangle_g=\{f\in S \mid f\neq g, \ \overline{f}\notin Irr(\{g\}) \},\ \ \ \
S_1=S\backslash \cup_{g\in S}\triangle_g.
\end{eqnarray*}
By Lemma \ref{lpredu}, we conclude that $S_1$ is a minimal Gr\"{o}bner--Shirshov basis for $I$.

For any $s\in S_1$, by Lemma \ref{lmgsb}, we have $s=s'+s''$, where $supp(s')\subset Irr(S_1\backslash \{s\}), \ s''\in Id(S_1\backslash \{s\})$. Let
$$
S_2=\{s' \mid s\in S_1\}.
$$
We  claim that
$S_2$ is a reduced Gr\"{o}bner--Shirshov basis for $I$. Indeed, it is clear that $S_2\subset Id(S_1)$.
For any $f\in Id(S_1)$, by Theorem \ref{cd2} and Lemma \ref{lmgsb}, $\overline{f}=\overline{\lfloor sb\rfloor_m}=\overline{\lfloor s'b\rfloor_m}$
for some normal $S\!_1$-polynomial $\lfloor sb\rfloor_m$ and normal $S\!_2$-polynomial $\lfloor s'b\rfloor_m$. According to  Lemma \ref{remark5.8},
we have $S_2$ is a Gr\"{o}bner--Shirshov basis for $I$.
Suppose there exists $s'\in S_2$
such that $supp(s') \nsubseteq Irr(S_2 \backslash \{s'\})$, i.e. there exists
$\lfloor a\rfloor_m\in supp(s')\subset Irr(S_1\backslash \{s\})$ but $\lfloor a\rfloor_m \notin Irr(S_2 \backslash \{s'\})$.
Then $\lfloor a\rfloor_m=\overline{\lfloor t'b\rfloor_m}$  for some normal $t'$-polynomial $\lfloor t'b\rfloor_m$,
where $\lfloor b\rfloor\in \lfloor X^*\rfloor, t'\in S_2 \backslash \{s'\}$, $t'=t-t''$,
$t\in S_1\backslash \{s\}$ and $t''\in Id(S_1\backslash \{t\})$.
By Lemma \ref{lmgsb}, $\overline{t}=\overline{t'}$. We must have $m=1$, $\lfloor b\rfloor\in \lfloor X^+\rfloor$,
$t'$ is strong and $t$ is not strong. Otherwise
$\lfloor tb\rfloor_m$ is a normal $S_1\backslash \{s\}$-polynomial and $\lfloor a\rfloor_m=\overline{\lfloor t'b\rfloor_m}=\overline{\lfloor tb\rfloor_m}$,
which contradicts our assumption. Then $\overline{t'}=\overline{t}=\lfloor c\rfloor_2$,
 where $\lfloor c\rfloor\in \lfloor X^+\rfloor, |c|=2$. Recall $r\!_{_t}:=t-\overline{t}$.
Hence $\overline{r\!_{_t}}=\lfloor c\rfloor_1=\overline{t''}$. There exist $f\in S_1\backslash \{t\}$ and a normal $f$-polynomial $\lfloor fd\rfloor_1$ such that $\overline{\lfloor fd\rfloor_1}=[c]_1$. It follows that $f$ is strong.
Since $\lfloor a \rfloor_m=\lfloor a \rfloor_1\in supp(s')$, by Lemma \ref{lmgsb}, we have
$\overline{f}<\overline{\lfloor fdb\rfloor_1}=\lfloor cb\rfloor_1=\overline{\lfloor t'b\rfloor}_1=\lfloor a \rfloor_1\leq \overline{s}$.
Hence $f\neq s$ and $\lfloor a\rfloor_1=\overline{\lfloor fdb\rfloor_1}$, where $\lfloor fdb\rfloor_1$ is a normal $S_1\backslash \{s\}$-polynomial.
This contradicts the fact that $\lfloor a\rfloor_1\in Irr(S_1\backslash \{s\})$. We thus get $supp(s') \subseteq Irr(S_2 \backslash \{s'\})$
for all $s'\in S_2$ and our claim holds.

Suppose that $T$ is a reduced Gr\"{o}bner--Shirshov basis for $I$. Let $\overline{s_0}=\min \overline{S_2}$ and
$\overline{r_0}=\min \overline{T}$, where $s_0\in S_2, r_0\in T$. By Theorem \ref{cd2} and Lemma \ref{lorder},
$\overline{s_0}=\overline{\lfloor r'a'\rfloor_p}\geq \overline{r'}\geq \overline{r_0}$ for some $r'\in T$ and
normal $T$-polynomial $\lfloor r'a'\rfloor_p$.
Similarly, $\overline{r_0}\geq \overline{s_0}$. Then $\overline{r_0}=\overline{s_0}$. We claim that $r_0=s_0$.
Otherwise, $0\neq r_0-s_0\in I$. We apply the above argument again, with replace $\overline{s_0}$ by $\overline{r_0-s_0}$,
to obtain that $\overline{r_0}> \overline{r_0-s_0}\geq\overline{r''}\geq \overline{r_0}$ for some $r''\in T$, a contradiction.
As both $T$ and $S_2$ are reduced Gr\"{o}bner--Shirshov bases, we have $S_2^{\overline{s_0}}=\{s_0\}=\{r_0\}=T^{\overline{r_0}}$.
Given any $\lfloor w\rfloor_m\in \overline{S_2}\cup \overline{T}$ with $\lfloor w\rfloor_m> \overline{r_0}$.
Assume that $S_2^{<\lfloor w\rfloor_m}=T^{<\lfloor w\rfloor_m}$.
To prove $T= S_2$, it is sufficient to show that $S_2^{\lfloor w\rfloor_m} \subseteq T^{\lfloor w\rfloor_m}$.
For any $s\in S_2^{\lfloor w\rfloor_m}$, we can see that
$\overline{s}=\overline{\lfloor ra''\rfloor_q}\geq \overline{r}$ for some $r\in T$ and normal $T$-polynomial $\lfloor ra''\rfloor_q$.
Now, we claim that $\lfloor w\rfloor_m=\overline{s}=\overline{r}$. Otherwise, $\lfloor w\rfloor_m=\overline{s}>\overline{r}$.
Then $r\in T^{<\lfloor w\rfloor_m}=S_2^{<\lfloor w\rfloor_m}$ and $r\in S_2\backslash \{s\}$.
But $\overline{s}=\overline{\lfloor ra''\rfloor_q}$, which contradicts the fact that $S_2$ is a reduced Gr\"{o}bner--Shirshov basis.
We next claim that $s=r\in T^{\lfloor w\rfloor_m}$. If $s\neq r$, then $0\neq s-r\in I$.
By Theorem \ref{cd2} and Lemma \ref{lorder}, $\overline{s-r}=\overline{\lfloor r_1u\rfloor_n}=\overline{\lfloor s_1v\rfloor_n}$ for some
normal $T$-polynomial $\lfloor r_1u\rfloor_n$ and normal $S_2$-polynomial $\lfloor s_1v\rfloor_n$ with
$\overline{r_1}, \overline{s_1}\leq \overline{s-r}<\overline{s}=\overline{r}$, where $r_1\in T, s_1\in S_2$.
This means that $s_1\in S_2\backslash \{s\}$ and $r_1\in T\backslash \{r\}$.
Noting that $\overline{s-r}\in supp(s)\cup supp(r)$, we may assume that $\overline{s-r}\in supp(s)$.
As $S_2$ is a reduced Gr\"{o}bner--Shirshov basis, we have $\overline{s-r}\in Irr(S_2\backslash \{s\})$,
which contradicts the fact that $\overline{s-r}=\overline{\lfloor s_1v\rfloor_n}$, where $s_1\in S_2\backslash \{s\}$.
Thus $s=r$. This shows that $S_2^{\lfloor w\rfloor_m} \subseteq T^{\lfloor w\rfloor_m}$.
\end{proof}

\begin{remark}\label{rrgsb}
Theorem \ref{trgsb} together with Buchberger--Shirshov algorithm gives a method to find the unique reduced Gr\"{o}bner--Shirshov basis $S^{red}$ for the ideal $I=Id(S)$, where $S$ is a monic subset of $Di[X]$. One can find $S^{red}$ by the following steps.

\begin{enumerate}
\item[(i)]\ Buchberger--Shirshov algorithm gives a Gr\"{o}bner--Shirshov basis $S_0:=S^{comp}$ for $I$.

\item[(ii)]\  Let
\begin{eqnarray*}
S_1=S_0\backslash \cup_{_{g\in S_0}}\triangle_g,   \ \mbox{where} \
\triangle_g=\{f\in S_0 \mid f\neq g, \ \overline{f}\notin Irr(\{g\})\}.
\end{eqnarray*}

Then $S_1$ is a minimal Gr\"{o}bner--Shirshov basis for $I$.

\item[(iii)]\ For each $s\in S_1$, by the same method as in the proof of Lemma \ref{l00}, $s$ has an expression $s=s'+s''$, where  $supp(s')\subset Irr(S_1\backslash \{s\})$, $s''\in Id(S_1\backslash \{s\})$. Then
$
 S^{red}=\{s' \mid s\in S_1\}
$
is the reduced Gr\"{o}bner--Shirshov basis for $I$.
\end{enumerate}
\end{remark}

\section{Polynomial dialgebra has a finite Gr\"{o}bner--Shirshov basis}

In this section, we use the Hilbert's basis theorem to prove that the polynomial diring $Di_R[X]$ is left Noetherian if $R$ is left Noetherian and $X$ is a finite set.
Furthermore, we show that  each ideal of the polynomial dialgebra $Di[X]$ has a finite Gr\"{o}bner--Shirshov basis, if $X$ is a finite set.

Throughout this section, $R$ is an associative ring with unit.
\begin{definition}
A \emph{diring} is a quaternary $(T,+, \vdash, \dashv)$  such that both
 $(T,+, \vdash)$ and $(T,+,\dashv)$ are associative rings with the identities (\ref{eq00})
in  $T$.
\end{definition}

\begin{definition}
Let $(D, \vdash, \dashv)$ be a disemigroup, $R$ an associative ring with unit and $T$ the free left $R$-module with an $R$-basis $D$. Then
$(T,+, \vdash, \dashv)$ is a diring with a natural way: for any $f=\sum_{i} r_if_i,\ g=\sum_{j} r_j'g_j\in T,\ r_i,r_j'\in R,\ f_i, g_j\in D$,
$$
f\vdash g:=\sum_{i,j} r_ir_j'(f_i\vdash g_j),\ \ \ \ f\dashv g:=\sum_{i,j} r_ir_j'(f_i\dashv g_j).
$$
Such a diring is called a \emph{disemigroup-diring} of $D$ over $R$.

For the free commutative disemigroup $Disgp[X]$ generated by a set $X$,
we denote $Di_R[X]$ the disemigroup-diring of $Disgp[X]$ over $R$
which is also called the  \emph{polynomial diring} over $R$.
In particular, $Di_\mathbf{k}[X]\ (=Di[X])$ is the free commutative dialgebra (\emph{polynomial dialgebra}) generated by $X$ when $\mathbf{k}$ is a field.

A \emph{left} (\emph{right}, resp.) \emph{ideal} $I$ of $Di_R[X]$ is an $R$-submodule of $Di_R[X]$ such that
$f\vdash g, f\dashv g\in I$ $(g\vdash f, g\dashv f \in I$, resp.) for any $f\in Di_R[X]$ and $g\in I$.
We call $I$  an \emph{ideal} of $Di_R[X]$ if $I$ is a left and right ideal.
\end{definition}

Clearly, $Di_R[X]$ is a commutative disemigroup-diring if and only if $R$ is a commutative ring.

Let $U$ be a semigroup, $R$ an associative ring and $RU$ the semigroup ring of $U$ over $R$. Note that $RU$ is a free left $R$-module with an $R$-basis $U$.

Thus, $R\lfloor X^*\rfloor$ is the ring of polynomials over $R$ in indeterminates $X$. It is known that $R\lfloor X^*\rfloor$ is left Noetherian if $R$ is left Noetherian and $X$ is a finite set.

\begin{theorem}
Let $X$ be a finite set. If $R$ is left Noetherian, then so is $Di_R[X]$.
\end{theorem}

\begin{proof}
Let $S=\{\lfloor a\rfloor_1\in \lfloor X^+ \rfloor_{_1} \mid |a|\geq 3\}$ and $I=Id(S)$ the ideal of $Di_R[X]$ generated by $S$. We first prove that $Di_R[X]/I$ is left Noetherian.
It suffices to show that any left ideal $H$ of $Di_R[X]/I$ has a finite set of generators. Since $X$ is finite,
$\{\lfloor a\rfloor_p\in \lfloor X^+ \rfloor_{_1}\cup \lfloor XX \rfloor_{_{2}}\mid |a|\leq 2\}$ is finite, say, $\{\lfloor a\rfloor_p\in \lfloor X^+ \rfloor_{_1}\cup \lfloor XX \rfloor_{_{2}}\mid |a|\leq 2\}=\{\lfloor a_1\rfloor_{p_1},\dots,\lfloor a_m\rfloor_{p_m}\}$ with
$\lfloor a_1\rfloor_{p_1}<\dots<\lfloor a_m\rfloor_{p_m}$, where $<$ is the deg-lex-center ordering  on $\lfloor X^+ \rfloor_{_1}\cup \lfloor XX \rfloor_{_{2}}$.
Then $Di_R[X]/I=R\lfloor a_1\rfloor_{p_1}+R\lfloor a_2\rfloor_{p_2}+\cdots+R\lfloor a_m\rfloor_{p_m}$.
For $j=1,2,\dots,m$, let $I_j$ be the set of elements $\alpha_j\in R$ such that there exists an element of the form
$$
f_j=\alpha_j\lfloor a_j\rfloor_{p_j}+\beta_{j-1}\lfloor a_{j-1}\rfloor_{p_{j-1}}+\cdots +\beta_1\lfloor a_1\rfloor_{p_1}\in H, \ \
\beta_1,\cdots,\beta_{j-1}\in R.
$$
It is easy to check that $I_j$ is a left ideal of $R$.
Hence each $I_j$ has a finite set of generators and
we may assume that
$I_j$ is generated by $\alpha_{j1}, \cdots, \alpha_{j{n_j}}$, where $n_j \in \mathbb{Z}^+$. Then for $1\leq i \leq n_j$,
we have polynomials $f_{ji}=\alpha_{ji}\lfloor a_j\rfloor_{p_j}+h_{ji}\in H$ where $\overline{h_{ji}}<\lfloor a_j\rfloor_{p_j}$.
Thus we obtain a finite set $T:=\{f_{ji}\in H \mid 1\leq j\leq m, 1\leq i \leq n_j\}$.

Now we claim that the set $T$ generates $H$. For any nonzero $f\in H$, we prove $f$ belongs to the left ideal of $Di_R[X]/I$ generated by $T$, by induction on $\overline{f}$. Assume that $f=\alpha \lfloor a_k\rfloor_{p_k}+f_1$, where $1\leq k\leq m$, $0\neq\alpha\in R$
and $\overline{f_1}<\lfloor a_k\rfloor_{p_k}$.
Thus $\alpha\in I_k$ and $\alpha=\Sigma_ir_i\alpha_{ki}$, where $r_i\in R$.
Clearly, $\Sigma_ir_if_{ki}\in H$ and $lt(\Sigma_ir_if_{ki})=\alpha \lfloor a_k\rfloor_{p_k}=lt(f)$.
It follows that $f_1=f-\Sigma_ir_if_{ki}\in H$ and $\overline{f_1}<\overline{f}=\lfloor a_k\rfloor_{p_k}$.
If $\lfloor a_k\rfloor_{p_k}=\min \{\overline{g}\mid 0\neq g\in H\}$,
then $f_1=0$ and $f=\Sigma_ir_if_{ki}$. Hence our assertion holds.
Suppose that $\lfloor a_k\rfloor_{p_k}>\min \{\overline{g}\mid g\in H\}$. By induction on $\overline{f}$, we obtain the assertion.

Now, for any ascending chain of left ideals of $Di_R[X]$
\begin{eqnarray*}\label{ch0}
L_1\subseteq  L_2\subseteq \cdots\subseteq L_i\subseteq \cdots,
\end{eqnarray*}
we have an ascending chain of left ideals of $Di_R[X]/I$
\begin{eqnarray*}
(L_1+I)/I\subseteq  (L_2+I)/I\subseteq \cdots\subseteq (L_i+I)/I\subseteq \cdots.
\end{eqnarray*}
Since $Di_R[X]/I$ is left Noetherian, it follows that there is a $p\in \mathbb{Z}^+$ such that $(L_p+I)/I=(L_{p+1}+I)/I=\cdots$. Therefore
$L_p+I=L_{p+1}+I=\cdots$. On the other hand, note that for any $f\in Di_R[X],\ h\in I$, we have $f\vdash h=f\dashv h$, in particular, in $I,\ ``\vdash"=``\dashv"$.
Thus, $I$ is also a left ideal of the associative ring $(R\lfloor X^*\rfloor,+,\dashv)$.
Then for the ascending chain of left ideals of $R\lfloor X^*\rfloor$
\begin{eqnarray*}
L_1 \cap I\subseteq  L_2\cap I \subseteq \cdots\subseteq L_i\cap I\subseteq \cdots,
\end{eqnarray*}
since $(R\lfloor X^*\rfloor,+,\dashv)$ is left Noetherian, there is an $l\in \mathbb{Z}^+$ such that $L_l \cap I=L_{l+1}\cap I=\cdots$.
Take $n=\max\{p,l\}$. We thus get
$
L_n=L_n+(L_n\cap I)=L_n+(L_{n+1}\cap I)=L_{n+1}\cap(L_n+ I)=L_{n+1}\cap(L_{n+1}+ I)=L_{n+1}=\dots.
$
This shows that $Di_R[X]$ is left Noetherian.
\end{proof}

\begin{corollary}\label{cn}
The polynomial dialgebra $Di[X]$ is Noetherian, if $X$ is a finite set.
\end{corollary}

\begin{theorem}\label{thfgsb}
Let $X$ be a finite set and $I$ an ideal of $Di[X]$. Then for any monomial-center ordering on $\lfloor X^+ \rfloor_{_1}\cup \lfloor XX \rfloor_{_{2}}$, $I$ has a finite Gr\"{o}bner--Shirshov basis.
\end{theorem}

\begin{proof}
By Theorem \ref{trgsb}, $I$ has a minimal Gr\"{o}bner--Shirshov basis $S$.
The leading monomial of a non-strong polynomial $f\in S$ is $\lfloor x_ix_j\rfloor_2$ for some $x_i,x_j\in X$ by Proposition \ref{pnotstrong}.
Since $X$ is finite, it follows that the set $\{\lfloor x_ix_j\rfloor_2\mid x_i,x_j\in X\}$ is finite.
As $S$ is a minimal Gr\"{o}bner--Shirshov basis, we conclude that the set $\{s\in S\mid s \ \mbox{is not strong} \}$ is finite.
If $S$ is infinite, then there exists an infinite sequence $s_1, s_2,\dots$, where each $s_i\in S$ is strong and $\overline{s_i} \neq \overline{s_j}$ if $i\neq j$.
For $ n=1,2,\dots$,
let $I_n$ be the ideal of $Di[X]$ generated by $T_n=\{\overline{s_1},\dots,\overline{s_n}\}$. Then $\overline{s_{n+1}}\not\in I_n$. Otherwise, $\overline{s_{n+1}}$ has an expression
$$
\overline{s_{n+1}}=\sum_{i=1}^n (\overline{s_i}\dashv f_i+ \overline{s_i}\vdash h_i +\alpha_i\overline{s_i}).
$$
where $f_i,h_i\in Di[X],\ \alpha_i\in \mathbf{k}$. This implies that either $\overline{s_{n+1}}=\overline{s_{i}}\vdash x$
for some $\overline{s_i} \in T_n$ with  $|\widetilde{s_i}|=1$ and $x\in X$, or $\overline{s_{n+1}}=\overline{s_{i}}\dashv \lfloor a\rfloor_1$ for some $\overline{s_i} \in T_n$ and
$\lfloor a\rfloor \in \lfloor X^+\rfloor$.
Because each $s_i$ is strong, we can conclude that $s_{n+1}=\overline{s_{i}}\vdash x= \overline{s_i\vdash x}$
or $\overline{s_{n+1}}=\overline{s_{i}}\dashv \lfloor a\rfloor_1=\overline{s_i\dashv \lfloor a\rfloor_1}$.
It follows that
$\overline{s_{n+1}}=\overline{\lfloor s_ib\rfloor_m}$ for some normal $s_i$-polynomial $\lfloor s_ib\rfloor_m$, which contradicts the fact that $S$ is a minimal Gr\"{o}bner--Shirshov basis.
Then we have an infinite sequence $I_1\subset I_2\subset\dots$ which contradicts that $Di[X]$ is Noetherian.
\end{proof}

\begin{remark}
Under the conditions in Theorem \ref{thfgsb}, in $Di[X]$, every minimal Gr\"{o}bner--Shirshov basis is finite. In particular, every reduced Gr\"{o}bner--Shirshov basis is finite.
\end{remark}

By Theorem \ref{thfgsb}, we immediately have

\begin{corollary}
If $X$ is a finite set, then each congruence on the commutative disemigroup $Disgp[X]$ is finitely generated.
\end{corollary}

\begin{proof}
Suppose that $\rho$ is a congruence on $Disgp[X]$ and $\rho$ is not finitely generated.
Then we can obtain a strictly increasing infinite chain of congruences on $Disgp[X]$
\begin{eqnarray*}
\rho_1\subsetneq \rho_2\subsetneq\cdots \subsetneq\rho_i \subsetneq \cdots.
\end{eqnarray*}
Let $S_i=\{\lfloor u\rfloor_m-\lfloor v\rfloor_n \mid (\lfloor u\rfloor_m,\lfloor v\rfloor_n)\in \rho_i,\lfloor u\rfloor_m>\lfloor v\rfloor_n\}, i=1,2,\cdots$.
It is easy to check that each $S_i$ is a Gr\"{o}bner--Shirshov basis in $Di[X]$ and
\begin{eqnarray*}
S_1\subsetneq S_2\subsetneq\cdots \subsetneq S_i \subsetneq \cdots.
\end{eqnarray*}
Assume that $\lfloor u\rfloor_m-\lfloor v\rfloor_n\in S_{i+1}\subseteq Id(S_{i+1})$ and $\lfloor u\rfloor_m-\lfloor v\rfloor_n\notin S_i$.
We claim that $\lfloor u\rfloor_m-\lfloor v\rfloor_n\notin Id(S_i)$. Otherwise, by Theorem \ref{cd2}, we may assume that
\begin{eqnarray*}
\lfloor u\rfloor_m-\lfloor v\rfloor_n=\sum_{j=1}^l\alpha_j(\lfloor u_j\rfloor_{m_j}-\lfloor v_j\rfloor_{n_j}),
\end{eqnarray*}
where each $\alpha_j\in \mathbf{k}, \ \lfloor u_j\rfloor_{m_j}-\lfloor v_j\rfloor_{n_j}\in S_i$ and
$\lfloor u_1\rfloor_{m_1}>\cdots >\lfloor u_l\rfloor_{m_l}$. This implies that $\alpha_1=1, \lfloor u\rfloor_m=\lfloor u_1\rfloor_{m_1}$.
Now we prove that $\lfloor u\rfloor_m-\lfloor v\rfloor_n\in S_i$ by induction on $\lfloor u\rfloor_m$.
If $\lfloor u\rfloor_{m}=\min\{\overline{s}\mid s\in Id(S_i)\}$,
then $\lfloor u\rfloor_m-\lfloor v\rfloor_n=\lfloor u_1\rfloor_{m_1}-\lfloor v_1\rfloor_{n_1}\in S_i$.
Otherwise, we can assume that $\lfloor v_1\rfloor_{n_1}-\lfloor v\rfloor_n=\lfloor u\rfloor_m-\lfloor v\rfloor_n-(\lfloor u_1\rfloor_{m_1}-\lfloor v_1\rfloor_{n_1})\neq 0$ and $\lfloor v_1\rfloor_{n_1}>\lfloor v\rfloor_n$. Then
\begin{eqnarray*}
\lfloor v_1\rfloor_{n_1}-\lfloor v\rfloor_n=\sum_{j=2}^l\alpha_j(\lfloor u_j\rfloor_{m_j}-\lfloor v_j\rfloor_{n_j}).
\end{eqnarray*}
By induction on $\lfloor u\rfloor_m$, $\lfloor v_1\rfloor_{n_1}-\lfloor v\rfloor_n\in S_i$.
It follows that $(\lfloor v_1\rfloor_{n_1},\lfloor v\rfloor_n)\in \rho_i$. Since $(\lfloor u_1\rfloor_{m_1},\lfloor v_1\rfloor_{n_1})\in \rho_i$,
we have $(\lfloor u\rfloor_{m},\lfloor v\rfloor_n)=(\lfloor u_1\rfloor_{m_1},\lfloor v\rfloor_n)\in \rho_i$
and $\lfloor u\rfloor_{m}-\lfloor v\rfloor_n\in S_i$. This contradicts our assumption that $\lfloor u\rfloor_m-\lfloor v\rfloor_n\notin S_i$.
Therefore $Id(S_i)\subsetneq Id(S_{i+1}), i=1,2,\cdots$, and we can obtain a infinite
properly ascending chain of ideals, which is contradicts the fact that $Di[X]$ is Noetherian.
\end{proof}

\section{Applications}

\subsection{Normal forms of commutative disemigroups}

Recall that $Disgp[ X]=\lfloor X^+ \rfloor_{_1}\cup \lfloor XX \rfloor_{_{2}}$ is the free commutative disemigroup generated by a set $X$. Then
for an arbitrary commutative disemigroup $D$, $D$ has an expression
$$
D=Disgp[ X| S]:=Disgp[ X]/\rho(S)
$$
for some set $X$ and $S\subseteq Disgp[ X]\times Disgp[ X]$, where $\rho(S)$ is the congruence of $Disgp[ X]$ generated by $S$.

It is natural to ask how to find normal forms of elements of the commutative disemigroup  $D=Disgp [X|S]$?

Let $>$ be a monomial-center ordering on $Disgp[ X]$ and
$S=\{(\lfloor u_i\rfloor_{m_i},\lfloor v_i\rfloor_{n_i})\mid \lfloor u_i\rfloor_{m_i}>\lfloor v_i\rfloor_{n_i},\ i\in I\}$.
Consider the commutative dialgebra $Di[X| S]$, where $S=\{\lfloor u_i\rfloor_{m_i}-\lfloor v_i\rfloor_{n_i}\mid  i\in I\}$.
By  Buchberger--Shirshov algorithm, we have a Gr\"{o}bner--Shirshov basis $S^{comp}$ in $Di[X]$ and $Id(S^{comp})=Id(S)$. It is clear that each element in $S^{comp}$ is of the form $\lfloor u\rfloor_m-\lfloor v\rfloor_n$, where $\lfloor u\rfloor_m>\lfloor v\rfloor_n,\
\lfloor u\rfloor_m, \lfloor v\rfloor_n\in Disgp[ X]$.

The following theorem is an application of Theorem \ref{cd2}.
\begin{theorem}\label{tdsg}
Let $>$ be a monomial-center ordering on $Disgp[ X]$ and $D=Disgp[X| S]$,
where $S=\{(\lfloor u_i\rfloor_{m_i},\lfloor v_i\rfloor_{n_i})\mid \lfloor u_i\rfloor_{m_i}>\lfloor v_i\rfloor_{n_i},\ i\in I\}$.
Then the set $Irr(S^{comp})$ is a set of normal forms of elements of  the commutative disemigroup  $D=Disgp[X| S]$, where $S^{comp}$ is a Gr\"{o}bner--Shirshov basis in $Di[X]$ obtained from $S$ by Buchberger--Shirshov algorithm.
\end{theorem}

\subsection{Word problem for commutative dialgebras}

In this subsection, we present algorithms  for computing a (reduced) Gr\"{o}bner--Shirshov basis.
By using these algorithms and Composition-Diamond lemma for commutative dialgebras,
we  prove that for finitely presented commutative dialgebras (disemigroups), the word problem and the problem
whether two ideals  of $Di[X]$ are identical are solvable.

Let~$f\in Di[X]$ be a nonzero polynomial
 and let~$S$ be a subset of~$Di[X]$. We say~$f$ is \textit{reducible} by~$S$
if there exist a polynomial~$g$ in~$S$ and a
diword~$\lfloor u\rfloor_m$ in~$supp(f)$ such that~$\lfloor u\rfloor_m=\overline{\lfloor gb\rfloor_m}$ for some normal $g$-polynomial~$\lfloor gb\rfloor_m$.  Moreover, if $f$ is reducible by $S$ and $\alpha$ is the coefficient of $\lfloor u\rfloor_m$, then
$f\rightarrow_{_S} f-\frac{\alpha}{lc(g)}\lfloor gb\rfloor_m$ is called a \textit{one-step-reduction} by $S$. In particular,
we say $f$ is \textit{partially reducible} by $S$ if there exists a polynomial~$g$ in~$S$ and~$\overline{f}=\overline{\lfloor gb\rfloor_m}$ for some normal $g$-polynomial $\lfloor gb\rfloor_m$.
 If an element~$r\in Di[X]$ is
obtained from $f$ by finitely many one-step-reductions by $S$ and $r$ is not reducible by $S$, then we say
that $r$ is  a \textit{remainder} of $f$ modulo $S$.

In what follows, we assume that $X$ is a finite set. The following algorithm is an analogue of the Buchberger's Algorithm for polynomial algebras in \cite{bu65}.
\begin{algorithm}\label{alg1}(Algorithm for computing Gr\"{o}bner--Shirshov bases in $Di[X]$)

$\mathbf{input}: F=\{f_1,\cdots, f_m\}\subseteq Di[X]$.

$\mathbf{output}:$ A Gr\"{o}bner--Shirshov basis $S$ for $Id(F)$.

 \ \ \ \ $S:=\{lc(f)^{-1}f \mid 0 \neq f\in F\}$

 \ \ \ \ $\mathcal{G}:=\{s,(s,g) \mid s,g\in S\}$

 \ \ \ \ $\mathbf{for}$ $u\in \mathcal{G}$ $\mathbf{do}$

   \ \ \ \ \ \ \ \ $\mathbf{if}$ there exists a composition of $u$ $\mathbf{then}$

   \ \ \ \ \ \ \ \ \ \ \ \ $h:=$ a remainder of the composition of $u$ modulo $S$

    \ \ \ \ \ \ \ \ \ \ \ \ \ \ \ \ $\mathbf{if}$ $h\neq 0$ $\mathbf{then}$

    \ \ \ \ \ \ \ \ \ \ \ \  \ \ \ \ \ \ \ \ $h:=lc(h)^{-1}h$

    \ \ \ \ \ \ \ \ \ \ \ \  \ \ \ \ \ \ \ \ $\mathcal{G}:=\mathcal{G} \cup \{h\}\cup \{(h,s) \mid s\in S\} \cup \{(s,h) \mid s\in S\}$

     \ \ \ \ \ \ \ \ \ \ \ \  \ \ \ \ \ \ \ \ $\mathcal{S}:=\mathcal{S} \cup \{h\}$

    \ \ \ \ \ \ \ \ \ \ \ \ \ \ \ \ $\mathbf{end}$ $\mathbf{if}$

    \ \ \ \ \ \ \ \ $\mathbf{end}$ $\mathbf{if}$

  \ \ \ \ $\mathbf{end}$ $\mathbf{do}$

$\mathbf{return}$ $S$
\end{algorithm}

\begin{lemma}\label{lfgsb}
Let $>$ be a monomial-center ordering on $\lfloor X^+ \rfloor_{_1}\cup \lfloor XX \rfloor_{_{2}}$.
Then Algorithm \ref{alg1} terminates after finitely many steps and returns a Gr\"{o}bner--Shirshov basis in $Di[X]$.
\end{lemma}

\begin{proof} (Correctness.) If Algorithm \ref{alg1} terminates after finitely many steps, then its correctness follows clearly from
Definition \ref{dcgsb}.

(Termination.) Suppose to the contrary that Algorithm \ref{alg1} does not terminate. Then, as the algorithm progress, we construct a set $S_i$ strictly larger than $S_{i-1}$ and obtain a strictly increasing infinite sequence
\begin{eqnarray*}
S_1\subsetneq S_2\subsetneq S_3 \subsetneq \cdots.
\end{eqnarray*}
Each $S_{i+1}$ is obtained from $S_i$ by adding $h\in Id(F)$ to $S_{i+1}$, where $h$ is a non-zero remainder of a composition of  elements  in $S_i$ modulo $S_i$. It is easy to see that $\overline{h}\in Irr(S_i)$. By the method similar to that in the proof of Theorem \ref{thfgsb}, we can obtain a infinite properly ascending chain of ideals, which is contradicts the fact that $Di[X]$ is Noetherian.
\end{proof}

\ \

By mimicking the method used in \cite{bu65}, we obtain the following algorithm.

\begin{algorithm}\label{alg2}(Algorithm for computing the reduced Gr\"{o}bner--Shirshov bases in $Di[X]$)

$\mathbf{input}:$ A Gr\"{o}bner--Shirshov basis $S=\{g_1,\cdots, g_n\}$ in $Di[X]$.

$\mathbf{output}:$ The reduced Gr\"{o}bner--Shirshov basis $RG$ for $Id(S)$.

  \ \ \ \ $C:=S$

 \ \ \ \ $\mathbf{for}$ $f\in C$ $\mathbf{do}$

   \ \ \ \  \ \ \ \ $S:=S \backslash \{f\}$

    \ \ \ \ \ \ \ \ $\mathbf{if}$ $f$ is not partially reducible by $S$ $\mathbf{then}$

    \ \ \ \ \ \ \ \ \ \ \ \ $S:=S\cup \{f\}$

     \ \ \ \ \ \ \ \ $\mathbf{end}$ $\mathbf{if}$

  \ \ \ \ $\mathbf{end}$ $\mathbf{do}$

   \ \ \ \ $RG:=\emptyset$

   \ \ \ \ $\mathbf{for}$ $h\in S$ $\mathbf{do}$

    \ \ \ \ \ \ \ \ $G:=S \backslash \{h\}$

    \ \ \ \ \ \ \ \ $r:=$ a remainder of $h$ modulo $G$

   \ \ \ \ \ \ \ \ $\mathbf{if}$ $r\neq 0$ $\mathbf{then}$ 
   
   \ \ \ \ \ \ \ \ \ \ \ \ $RG:=RG\cup \{r\}$
   
    \ \ \ \ \ \ \ \  $\mathbf{end}$ $\mathbf{if}$
   
   \ \ \ \ $\mathbf{end}$ $\mathbf{do}$

  $\mathbf{return}$ $RG$
\end{algorithm}

The following lemma follows from Theorem \ref{trgsb} directly.

\begin{lemma}\label{lfrgsb}
Let $>$ be a monomial-center ordering on $\lfloor X^+ \rfloor_{_1}\cup \lfloor XX \rfloor_{_{2}}$.
Then Algorithm \ref{alg2} terminates after finitely many steps and returns the reduced Gr\"{o}bner--Shirshov basis in $Di[X]$.
\end{lemma}

By Theorems $\ref{trgsb}$ and $\ref{thfgsb}$ and Algorithm $\ref{alg2}$, it follows that
\begin{corollary}\label{cfrgsb}
Let $X$ be a finite set, $I$ an ideal of $Di[X]$ and $>$ a monomial-center ordering on $\lfloor X^+ \rfloor_{_1}\cup \lfloor XX \rfloor_{_{2}}$.
Then there exists a unique finite reduced Gr\"{o}bner--Shirshov basis for I.
\end{corollary}

\begin{theorem}\label{tword}
The word problem for any finitely presented commutative dialgebra is algorithmically solvable.
\end{theorem}

\begin{proof} Let $D=Di[X|S]$ be a finitely presented commutative dialgebra. By Algorithm \ref{alg1}, we can obtain
a Gr\"{o}bner--Shirshov basis $G$ from $S$. For any $f\in Di[X]$,
by Theorem \ref{cd2}, $f\in Id(S)$ if and only if the remainder of $f$ modulo $G$ is equal to zero.
\end{proof}

\begin{theorem}\label{tideal}
Let $X$ be a finite set, $S$ and $T$ two finite subsets in $Di[X]$. The problem whether $Id(S)$ and $Id(T)$ are identical is algorithmically solvable.
\end{theorem}

\begin{proof} By Algorithms \ref{alg1} and \ref{alg2}, we can obtain the finite reduced Gr\"{o}bner--Shirshov bases $G_1$ for $Id(S)$
and $G_2$ for $Id(T)$ from $S$ and $T$, respectively. By Corollary \ref{cfrgsb}, $Id(S)=Id(T)$ if and only if $G_1=G_2$.
\end{proof}

\section{Lifting commutative Gr\"{o}bner--Shirshov bases to associative ones}

Let $Di\langle X\rangle$ be the free dialgebra over a field $\mathbf{k}$ generated by a set $X$. Recall that
$X^+$ is the free semigroup generated by $X$ without unit. Write
$$
[X^+]_\omega:=\{[u]_m \mid u\in X^+, m\in \mathbb{Z}^+, 1\leq m \leq |u|\}.
$$
It is well known that $[X^+]_\omega$ is a $\mathbf{k}$-basis of $Di\langle X\rangle$.

We first recall from the following result
concerning the Gr\"{o}bner--Shirshov bases theory for dialgebras in \cite{CZ}.

\begin{lemma}\citep[Theorem 4.4]{CZ}\label{cd}
Let $S$ be a monic subset of $Di\langle X\rangle$,
$>$ a monomial-center ordering on $[X^+]_\omega$ and
$Id(S)$ the ideal of $Di\langle X\rangle$ generated by $S$. Then $S$ is a Gr\"{o}bner--Shirshov basis in $Di\langle X\rangle$ if and only if
the set
$$
Irr_A(S)=\{[u]_n\in [X^+]_\omega\mid [u]_n\neq \overline{[asb]_m} \ \mbox{ for any normal } S\mbox{-polynomial}\ [asb]_m \ in \ Di\langle X\rangle\}
$$
is a $\mathbf{k}$-basis of the quotient dialgebra $Di\langle X| S\rangle:=Di\langle X\rangle/Id(S)$.
\end{lemma}

Now, we construct a Gr\"{o}bner--Shirshov basis in $Di\langle X\rangle$ by
lifting a Gr\"{o}bner--Shirshov basis in $Di[X]$ and adding some relations.

Let $X=\{x_i\mid i\in I\}$  be a well-ordered set. Note that
$\lfloor X^+\rfloor_1\cup \lfloor XX \rfloor_{_{2}}\subseteq [X^+]_\omega$.
We consider the natural map $\gamma:[X^+]_\omega\rightarrow \lfloor X^+\rfloor_1\cup \lfloor XX \rfloor_{_{2}}$ defined by, for any $[a]_m\in [X^+]_\omega$,
\begin{displaymath}
 \gamma([a]_m)=
\begin{cases}
\text{ $\lfloor a\rfloor_2$, if $|a|=2$ and $m=2$,} \\
\text{ $\lfloor a\rfloor_1$, otherwise,}
\end{cases}
\end{displaymath}
 and define a $\mathbf{k}$-linear map $\delta: Di[X]\rightarrow Di\langle X\rangle$
 which is given by
 $$
 \sum_i\alpha_i\lfloor a_i\rfloor_{m_i} \mapsto \sum_i\alpha_i\lfloor a_i\rfloor_{m_i},\ \ \ \ \alpha_i\in \mathbf{k},\
\lfloor a_i\rfloor_{m_i}\in  \lfloor X^+\rfloor_1\cup \lfloor XX \rfloor_{_{2}}.
 $$
Let $\succ$ be any monomial-center ordering on $\lfloor X^+\rfloor_1\cup \lfloor XX \rfloor_{_{2}}$ and $>_{d}$ the deg-lex-center ordering on $[X^+]_\omega$.
For any $[a]_m,[b]_n\in [X^+]_\omega$, define
\begin{equation}\label{liftorder} [a]_m>[b]_n \ \Leftrightarrow
\begin{cases}
\text{ $ \gamma([a]_m)\succ \gamma([b]_n)$,} \mbox{ or}\\
\text{ $\gamma([a]_m)=\gamma([b]_n),\  [a]_m>_d[b]_n$.}
\end{cases}
\end{equation}

It is clear that $>$ is a monomial-center ordering on $[X^+]_\omega$.
In what follows, a monomial-center ordering $\succ$ on $\lfloor X^+\rfloor_1\cup \lfloor XX \rfloor_{_{2}}$ and
the ordering $>$ on $[X^+]_\omega$ defined as (\ref{liftorder}) will be used unless otherwise stated.

For any $\lfloor a\rfloor_m,\lfloor b\rfloor_n\in \lfloor X^+\rfloor_1\cup \lfloor XX \rfloor_{_{2}}$,
if $\lfloor a\rfloor_m\succ\lfloor b\rfloor_n$, then
$\delta(\lfloor a\rfloor_m)>\delta(\lfloor b\rfloor_n)$. Thus, for any $f\in Di[X]$,
$\delta(\overline{f})=\overline{\delta(f)}$.
Let $s\in Di[X]$, $\lfloor sb\rfloor_m$  be a normal $s$-polynomial in $Di[X]$.
Then $\overline{\delta(\lfloor sb\rfloor_m)}
=\delta(\overline{\lfloor sb\rfloor_m})=\delta(\lfloor \widetilde{s}b\rfloor_m)$.

Let  $S$ be a monic subset of $Di[X]$, $\lfloor v\rfloor =\lfloor x_{i_1}x_{i_2}\cdots x_{i_r}\rfloor\in \lfloor X^+\rfloor, \ x_{i_l}\in X,\ 1\leq l\leq r$.
We set
$$
U(\lfloor v\rfloor):=\lfloor\{x_j\in X \mid i_1<j<i_r \}^*\rfloor,
$$
$$
G_S=:\{\delta(\lfloor sc\rfloor_m)\mid  s\in S, \lfloor c\rfloor\in U(\widetilde{s})\mbox{ and }\ \lfloor sc\rfloor_m\ \mbox{ is normal }S\mbox{-polynomial in }
Di[X]\}.
$$
For example,
let $\lfloor u\rfloor=x_1x_2x_2x_3x_4x_5x_6x_7x_7$ and $\lfloor v \rfloor=x_2x_6$.
Then $\lfloor u\rfloor=x_1x_2\lfloor v x_3x_4x_5\rfloor x_7x_7$, where $x_3x_4x_5 \in U(\lfloor v\rfloor)$.
That is to say, $\lfloor u\rfloor$ can split into three parts by $\lfloor v \rfloor$ and
the center part is $\lfloor v a \rfloor$ for some $\lfloor a\rfloor \in U(\lfloor v\rfloor)$.

Throughout the following proof, we follow the notation used in \cite{CZ}.

\begin{lemma}\label{llift}
Let $S$ be a monic subset of $Di[X]$ and $Y=\lfloor X^+\rfloor_1\cup \lfloor XX \rfloor_{_{2}}$.
Then
$$
Y\cap ([X^+]_\omega\backslash Irr_A(G_S))=Y\backslash \delta(Irr(S)).
$$
\end{lemma}

\begin{proof}
Let $[w]_n \in [X^+]_\omega$. Then
\begin{align*}
&[w]_n\in Y\cap ([X^+]_\omega\backslash Irr_A(G_S)) \\
\Leftrightarrow\ &
 \mbox{There exists a normal}\ G_S\mbox{-polynomial}\ [agb]_n \ \mbox{in}
 \ Di\langle X\rangle,  \mbox{ where } a,b\in X^*,\ g\in G_S,\ \mbox{say},\
 \\
  & g=\delta(\lfloor sc\rfloor_m)  \ \mbox{for some normal} \ S\mbox{-polynomial} \ \lfloor sc\rfloor_m\ \ \mbox{in} \ Di[X],\  s\in S,\ \lfloor c\rfloor\in U(\widetilde{s}),\ \mbox{such that} \\ &[w]_n=\overline{[a\delta(\lfloor sc\rfloor_m)b]_n}=[a\delta(\lfloor \widetilde{s}c\rfloor_m)b]_n\in Y.\ \mbox{Note that if } s \mbox{ is not strong, then } c \mbox{ is empty}\\
   &\mbox{and } \bar s=\lfloor xx'\rfloor_2,\ x,x'\in X,\ x\leq x'.\ \mbox{If this is the case, then } a,b,c \mbox{ are all empty. }\\
\Leftrightarrow\ &\mbox{There exist}\ s\in S,\ \lfloor c\rfloor\in U(\widetilde{s}), \ \lfloor a\rfloor,\lfloor b\rfloor\in \lfloor X^*\rfloor,\
\lfloor a\lfloor \widetilde{s}c\rfloor b\rfloor\in \lfloor X^+\rfloor\ \mbox{and a normal}\ S\mbox{-polynomial}
\\
&\lfloor sacb\rfloor_n \ \mbox{in} \ Di[X] \ \mbox{such that}\ [w]_n=\delta(\lfloor a\lfloor \widetilde{s}c\rfloor b\rfloor_n)=\delta(\overline{\lfloor sacb\rfloor_n})\in Y. \\
\Leftrightarrow\ &[w]_n\in Y\backslash \delta(Irr(S)).
\end{align*}
\end{proof}

\begin{theorem}
Let $X=\{x_i\mid i\in I\}$ be a well-ordered set, $S$ be a Gr\"{o}bner--Shirshov basis in
$Di[X]$, and $W$ consist of the following polynomials in $Di\langle X\rangle$:
\begin{align*}
&[x_ix_j]_2-[x_jx_i]_2,\ \ [x_ix_j]_1-[x_jx_i]_1, & (&i,j\in I,\ i>j),\\
&[x_ix_jx_k]_2-[x_ix_jx_k]_1, \ \ [x_ix_jx_k]_3-[x_ix_jx_k]_1, & (&i,j,k\in I, \ i\leq j \leq k).
\end{align*}
Then with the monomial-center ordering on $[X^+]_\omega$ defined as (\ref{liftorder}), the set
 $G_S\bigcup W$
is a Gr\"{o}bner--Shirshov basis in $Di\langle X\rangle$.
\end{theorem}

\begin{proof}
Let $T=G_S\bigcup W$ and
$Y=\lfloor X^+\rfloor_1\cup \lfloor XX \rfloor_{_{2}}$.
It is clear that $Irr_A(W)=Y$.
By Theorem \ref{cd2}, $Irr(S)$ is a $\mathbf{k}$-basis of the commutative dialgebra $Di[X|S]$.

Noting that $Di\langle X|T \rangle$ is a commutative dialgebra, we have a homomorphism
$$
\theta:\ Di[X]\rightarrow Di\langle X|T \rangle, \ \lfloor a\rfloor_{m}\mapsto \lfloor a\rfloor_{m}+Id(T),\ \lfloor a\rfloor_{m}\in Y,
$$
which is induced by $
X\rightarrow Di\langle X|T \rangle, \ x\mapsto x+Id(T)
$. Since $\theta(S)=0$, we have a homomorphism
\begin{eqnarray*}
\sigma:\ Di[X|S]\rightarrow Di\langle X|T \rangle, \
\sum_i\alpha_i\lfloor a_i\rfloor_{m_i}+Id(S)
\mapsto \delta(\sum_i\alpha_i\lfloor a_i\rfloor_{m_i})+Id(T),  \ \alpha_i\in \mathbf{k},\ \lfloor a_i\rfloor_{m_i}\in Y.
\end{eqnarray*}

On the other hand, the following homomorphism
$$
\xi:\ Di\langle X \rangle\rightarrow Di[X|S], \ [a]_m\mapsto \gamma([a]_m)+Id(S),\ [a]_m\in [X^+]_{\omega},
$$
is induced by $
X\rightarrow Di[X|S], \ x\mapsto x+Id(S)
$.  Consider the natural homomorphism
$$
\eta:\ Di\langle X \rangle\rightarrow Di\langle X|T \rangle, \ [a]_m\mapsto [a]_m+Id(T),\ [a]_m\in [X^+]_{\omega}.
$$
Since $\xi(T)=0$, we have a homomorphism $\zeta:\ Di\langle X|T \rangle \rightarrow Di[X|S]$ such that $\zeta\eta=\xi$ and $\zeta$ is exactly the inverse of $\sigma$.
This shows that $\sigma$ is an isomorphism.
Thus $\delta(Irr(S))=\{\delta(\lfloor a\rfloor_n)\mid \lfloor a\rfloor_n\in Irr(S)\}$ is a $\mathbf{k}$-basis of the dialgebra
$Di\langle X|T \rangle$. Now, for any $[w]_n \in [X^+]_\omega$, by Lemma \ref{llift}, we have
\begin{align*}
[w]_n\notin Irr_A(T) &\Leftrightarrow [w]_n\notin Irr_A(W),\mbox{or}\ [w]_n\in Irr_A(W)\ \mbox{and} \  [w]_n\notin Irr_A(G_S)\\
&\Leftrightarrow [w]_n\notin Y,\mbox{or}\ [w]_n\in Y\cap ([X^+]_\omega\backslash Irr_A(G_S))\\
&\Leftrightarrow [w]_n\notin Y,\mbox{or}\ [w]_n\in Y\backslash \delta(Irr(S)) \\
&\Leftrightarrow [w]_n\notin \delta(Irr(S)).
\end{align*}
Hence, $\delta(Irr(S))=Irr_A(T)$. It follows that $Irr_A(T)$ is a $\mathbf{k}$-basis of the dialgebra $Di\langle X|T \rangle$.
By Lemma \ref{cd}, $T$ is a Gr\"{o}bner--Shirshov basis in $Di\langle X\rangle$.
\end{proof}

\ \

\noindent{\bf Acknowledgement}
We wish to express our thanks to the referee for helpful suggestions and comments.

\ \

\bibliographystyle{elsarticle-harv}
% Include the ".bib" file (generated by bibtex) right here.

\end{document}